\documentclass[a4paper]{amsart}
\usepackage{graphicx}
\usepackage[all]{xy}
\theoremstyle{plain}
  \newtheorem{thm}{Theorem}[section]
  \newtheorem{lem}[thm]{Lemma}
  \newtheorem{cor}[thm]{Corollary}
  \newtheorem{prop}[thm]{Proposition}

\theoremstyle{definition}
  
  \newtheorem{ex}[thm]{Example}

\theoremstyle{remark}
  \newtheorem{rem}[thm]{Remark}
  \newtheorem*{ack}{Acknowledgments}
\newcommand{\Z}{\mathbb{Z}}
\newcommand{\cyclic}[1]{\Z/#1\Z}
\newcommand{\cyclicproduct}[2]{(\cyclic{#1})^{#2}}

\newcommand{\trivial}{\emptyset}
\newcommand{\zerocomp}{\emptyset_P}
\newcommand{\zerophrase}{\emptyset_M}
\newcommand{\xType}[1]{\lvert#1\rvert}
\newcommand{\abs}[1]{\lvert#1\rvert}

\newcommand{\cvector}[1]{\vec{#1}}

\newcommand{\gvector}[1]{\overrightarrow{#1}}

\newcommand{\SoAll}{V}
\newcommand{\SoDiag}{U}
\DeclareMathOperator{\rank}{rank}
\DeclareMathOperator{\hr}{hr}

\DeclareMathOperator{\nc}{nc}

\numberwithin{equation}{section}
\allowdisplaybreaks
\begin{document}
\title{Factorization of homotopies of nanophrases}
\author{Andrew Gibson}
\address{
Department of Mathematics,
Tokyo Institute of Technology,
Oh-okayama, Meguro, Tokyo 152-8551, Japan
}
\email{gibson@math.titech.ac.jp}
\date{\today}
\begin{abstract}
Homotopy on nanophrases is an equivalence relation defined using some data
 called a homotopy data triple.
We define a product on homotopy data triples.
We show that any homotopy data triple can be factorized into a product
 of prime homotopy data triples and this factorization is unique up to
 isomorphism and order.
If a homotopy data triple is composite, we show that equivalence of
 nanophrases under the corresponding homotopy can be calculated just by
 using the homotopies given by its prime factors.
\end{abstract}
\thanks{This work was supported by a Scholarship from the Ministry of
Education, Culture, Sports, Science and Technology of Japan.} 
\maketitle
\section{Introduction}
A word is a sequence of letters.
If every letter that appears in a word appears exactly twice, then the
word is a Gauss word.
Let $\alpha$ be a fixed set.
A nanoword over $\alpha$ is a Gauss word paired with a map from the set
of letters appearing in this word to $\alpha$.
A Gauss phrase is a sequence of words such that their concatenation
forms a Gauss word.
A nanophrase over $\alpha$ is a Gauss phrase paired with a map from the
set of letters appearing in this phrase to $\alpha$.
Turaev defined nanowords in \cite{Turaev:Words} and nanophrases in
\cite{Turaev:KnotsAndWords}.
\par
For nanowords and nanophrases, 
Turaev defined moves which are determined by an involution on $\alpha$
called $\tau$ and a subset of $\alpha \times \alpha \times \alpha$
called $S$.
We call the triple $(\alpha, \tau, S)$ a homotopy data triple.
Fixing a homotopy data triple, the moves generate an equivalence
relation on nanowords and nanophrases over $\alpha$ called homotopy.
Different homotopy data triples can give different equivalence
relations.
\par
Two homotopy data triples, $(\alpha,\tau,S)$ and
$(\alpha^\prime,\tau^\prime,S^\prime)$ are isomorphic if there is a
bijection from $\alpha$ to $\alpha^\prime$ which transforms $\tau$ into
$\tau^\prime$ and $S$ into $S^\prime$.
Isomorphic homotopy data triples give equivalent homotopies
\cite{Turaev:KnotsAndWords}.
\par
We define a product on homotopy data triples.
A homotopy data triple is \emph{composite} if it can be represented
as a non-trivial product of homotopy data triples.
If not, the homotopy data triple is \emph{prime}.
We show that any homotopy data triple can be represented uniquely, up to
order and isomorphism, as a product of prime homotopy data triples
(Proposition~\ref{prop:uft}).
\par
In this paper, our main aim is to show that for any composite homotopy
data triple, the study of the homotopy it gives can be reduced to the
study of the homotopies given by its prime factors.
\par 
Let $(\alpha, \tau, S)$ be a composite homotopy data triple.
Let $\mathcal{P}_R(\alpha)$ be the set of nanophrases over $\alpha$
satisfying the following conditions:
(1) each component is associated with a prime factor of
$(\alpha, \tau,S)$ and any letters in that component map to the factor;
and (2) adjacent components are associated with different prime factors.
Any nanoword over $\alpha$ can be uniquely split into a nanophrase in
$\mathcal{P}_R(\alpha)$ such that each component is non-empty.
This gives a map from nanowords over $\alpha$ to
$\mathcal{P}_R(\alpha)$ called the decomposing map.
\par
We define two equivalence relations on $\mathcal{P}_R(\alpha)$.
The equivalence relation $\sim_K$ is a restriction of homotopy to
$\mathcal{P}_R(\alpha)$.
It can be defined entirely in terms of the homotopies
associated with the prime factors of $(\alpha, \tau, S)$.
The equivalence relation $\sim_P$ is generated by $\sim_K$, reductions
(moves on nanophrases which allow us to remove empty components) and
their inverses.
\par
The key result of this paper is that there is a bijection between the
set of homotopy classes of nanowords over $\alpha$ and the set of
equivalence classes of $\mathcal{P}_R(\alpha)$ under $\sim_P$
(Theorem~\ref{thm:bijection}).
\par
We use this bijection to define a homotopy invariant of
nanowords.
Given a nanoword over $\alpha$, we apply the decomposing map to get
a nanophrase $p$ in $\mathcal{P}_R(\alpha)$.
We say that $p$ is reducible if it is equivalent under $\sim_K$ to a
nanophrase $q$ with an empty component.
By applying a reduction to $q$ we get a nanophrase which is equivalent
to $p$ under $\sim_P$ but has less components.
By repeating this process we will eventually get an irreducible
nanophrase.
We show that, irrespective of how we make the reductions, the
irreducible nanophrase that we obtain is unique up to equivalence under
$\sim_K$ (Proposition~\ref{prop:reduction-wd}).
In fact, this gives a complete homotopy invariant of the original
nanoword (Theorem~\ref{thm:invariance}). 
We generalize these results to nanophrases in
Section~\ref{sec:nanophrase-inv}.
\par
In general, calculating this invariant can be hard, because it is not
always easy to tell whether or not a nanophrase is irreducible.
Moreover, using the invariant to show that different nanowords are not
homotopic entails showing that their irreducible nanophrases are not
equivalent under $\sim_K$ which, in general, is also difficult.
However, the calculation of the invariant of a nanoword and
determination of the equivalence of this invariant for different
nanowords can be achieved just by using the homotopies given by the prime factors
of $(\alpha, \tau, S)$.
\par
A homotopy of nanophrases over $\alpha$ is \emph{equality decidable} if
there exists a finite time algorithm which given any two nanophrases
over $\alpha$, determines whether or not they are homotopic.
The homotopy is \emph{reduction decidable} if there exists a
a finite time algorithm which given any nanophrase over $\alpha$,
determines whether or not it is reducible.
\par
In the case where $S$ is the empty set, for any $\alpha$ and $\tau$, the
homotopy given by $(\alpha, \tau, S)$ is reduction and equality
decidable (Proposition~\ref{prop:emptys-decidable}).
In all other cases, determining whether a homotopy is reduction or equality
decidable is an open problem.
On the other hand, we use our complete invariant of nanowords and
nanophrases to show the following fact.
Let $(\alpha,\tau,S)$ be a composite homotopy data
 triple and let its prime factors be denoted by
 $(\alpha_i,\tau_i,S_i)$.
If the homotopies given by each $(\alpha_i,\tau_i,S_i)$ are all
reduction and equality decidable, the homotopy given by
$(\alpha,\tau,S)$ is reduction and equality decidable
(Theorem~\ref{thm:phrasedecidability}).
\par
In Section~\ref{sec:detecting} we examine several invariants which can
be used to give sufficient conditions for irreducibility.
Two of the invariants we consider, $\SoDiag$ and $\SoAll$, are new.
Both invariants are generalizations of our $S_o$ invariant for Gauss
phrases \cite{Gibson:gauss-phrase}.
The $\SoAll$ invariant is defined for all homotopies but $\SoDiag$ is
just defined for homotopies where $S$ is diagonal (that is, $S$ is of
the form $\{(a,a,a) \; | \; a \in \alpha\}$).
The $\SoDiag$ invariant is a generalization to nanophrases of Turaev's
self-linking function for nanowords \cite{Turaev:Words}.
Having written this paper we discovered that Fukunaga had independently  
generalized our $S_o$ invariant \cite{Fukunaga:gen-app}.
We give a proof that Fukunaga's invariant and our $\SoDiag$ invariant
are equivalent in Section~\ref{subsec:fukunagaequiv}.
\begin{ack}
The author is very grateful for all the help, advice and encouragement
 given to him by his supervisor Hitoshi Murakami.
He would like to thank Tomonori Fukunaga for several useful
 conversations.
The author is also very grateful for advice he received from Vladimir
 Turaev.
\end{ack}
\section{Nanowords and nanophrases}\label{sec:definitions}
Nanowords and nanophrases were defined by Turaev in \cite{Turaev:Words}
and \cite{Turaev:KnotsAndWords}.
In this section we briefly recall various definitions from those papers.
\par
An \emph{alphabet} is a finite set and its elements are called
\emph{letters}.
For any positive integer $m$, let $\hat{m}$ denote the set
$\{1, 2, \dotsc , m\}$.
A \emph{word} on an alphabet $\mathcal{A}$ of length $m$ is a map from
$\hat{m}$ to $\mathcal{A}$.
Informally, we can think of a word $w$ on $\mathcal{A}$ is a finite
sequence of letters in $\mathcal{A}$ and we will usually write words in
this way.
So, for example, $ABCA$ is a word of length $4$ on $\{A,B,C,D\}$.
We can write this word as a map from $\hat{4}$ to $\{A,B,C,D\}$,
where $1$ maps to $A$, $2$ maps to $B$, $3$ maps to $C$ and $4$ maps to
$A$.
The empty word of length $0$ on any alphabet is written $\trivial$.
\par
An $n$-component \emph{phrase} on an alphabet $\mathcal{A}$ is a
sequence of $n$ words on $\mathcal{A}$.
We write phrases as a sequence of words, separated by `$|$'
symbols.
For example, $A|BCA|CD$ is a $3$-component phrase on the alphabet
$\{A,B,C,D\}$.
Note that in \cite{Turaev:KnotsAndWords}, Turaev always encloses phrases
in brackets but in this paper we omit them.
\par
A \emph{Gauss word} on an alphabet $\mathcal{A}$ is a word on $\mathcal{A}$
such that every letter in $\mathcal{A}$ appears exactly twice in the
word.
A \emph{Gauss phrase} on an alphabet $\mathcal{A}$ is a phrase on
$\mathcal{A}$ such that every letter in $\mathcal{A}$ appears exactly
twice in the phrase.
Equivalently, a Gauss phrase on $\mathcal{A}$ is a phrase on
$\mathcal{A}$ such that the concatenation of the words appearing in the
phrase is a Gauss word on $\mathcal{A}$.
By definition, the single component appearing in a $1$-component Gauss
phrase is a Gauss word.
\par
Let $\alpha$ be a non-empty set.
An $\alpha$-alphabet is an alphabet $\mathcal{A}$ which has an
associated map from $\mathcal{A}$ to $\alpha$.
This map is called a \emph{projection}.
Following Turaev, we will often write $\xType{X}$ for the image of $X$
under the projection, where $X$ is a letter in $\mathcal{A}$.
An \emph{isomorphism} of $\alpha$-alphabets $\mathcal{A}_1$ and
$\mathcal{A}_2$ is a bijection $f$ from $\mathcal{A}_1$ to
$\mathcal{A}_2$ such that $\xType{f(X)}$ is equal to $\xType{X}$ for all
letters $X$ in $\mathcal{A}_1$.
\par
A {nanoword over $\alpha$} is a pair $(\mathcal{A},w)$ where
$\mathcal{A}$ is an $\alpha$-alphabet and $w$ is a Gauss word on
$\mathcal{A}$.
We write $\mathcal{N}(\alpha)$ for the set of nanowords over $\alpha$.
\par
An $n$-component \emph{nanophrase over $\alpha$} is a pair
$(\mathcal{A},p)$ where $\mathcal{A}$ is an $\alpha$-alphabet and 
$p$ is an $n$-component Gauss phrase on $\mathcal{A}$.
There is a unique nanophrase with $0$ components which we write
$\zerocomp$.
Note that in \cite{Turaev:KnotsAndWords} Turaev uses $\trivial$ to
denote an $0$-component nanophrase and $(\trivial)$ to represent a
nanophrase with a single empty component.
We always use $\trivial$ to represent an empty word and therefore,
according to context, it also represents an empty Gauss word or a
nanophrase with a single empty component.
\par
We write $\mathcal{P}_n(\alpha)$ for the set of $n$-component
nanophrases over $\alpha$
and write $\mathcal{P}(\alpha)$ for the set of all nanophrases over
$\alpha$.
Thus, $\mathcal{P}(\alpha)$ is the union of the
$\mathcal{P}_n(\alpha)$ with $n$ running from $0$ to infinity.
\par
There is a natural bijection between nanowords over $\alpha$ and
$1$-component nanophrases over $\alpha$ induced by mapping a Gauss word
$w$ to a $1$-component Gauss phrase containing a single component $w$.
In this paper, we will consider nanowords over $\alpha$ and
$1$-component nanophrases over $\alpha$ to be the same.
Thus $\mathcal{N}(\alpha)$ and $\mathcal{P}_1(\alpha)$ represent the
same set.
\par
Two nanophrases over $\alpha$, $(\mathcal{A}_1,p_1)$ and
$(\mathcal{A}_2,p_2)$, are \emph{isomorphic} if there exists an
isomorphism of $\alpha$-alphabets $f$ from $\mathcal{A}_1$ to
$\mathcal{A}_2$ such that $f$
applied letterwise to the $i$th component of $p_1$ gives the $i$th
component of $p_2$ for all $i$.
\par
Let $\tau$ be an involution on $\alpha$ (that is, $\tau(\tau(a))$ is equal
to $a$ for all $a$ in $\alpha$).
Let $S$ be a subset of $\alpha \times \alpha \times \alpha$.
We say that $S$ is \emph{diagonal} if
$S = \{(a,a,a) \; | \; a \in \alpha \}$.
\par
Turaev defined \emph{homotopy moves} for nanophrases on $\alpha$.
In moves on nanowords, the lower case letters $x$, $y$, $z$ and $t$
represent arbitrary sequences of letters, possibly including one or more
`$|$' symbols, so that the phrases on each side of the move are Gauss
phrases. The moves are
\par
\quad
move H1: for any $\xType{A}$,
\par
\quad\quad\quad
\quad\quad\quad
$(\mathcal{A},xAAy) \longleftrightarrow (\mathcal{A}-\{A\},xy)$
\par
\quad
move H2: if $\tau(\xType{A}) = \xType{B}$,
\par
\quad\quad\quad
\quad\quad\quad
$(\mathcal{A},xAByBAz) \longleftrightarrow (\mathcal{A}-\{A,B\},xyz)$ 
\par
\quad
move H3: if $(\xType{A},\xType{B},\xType{C}) \in S$,
\par
\quad\quad\quad
\quad\quad\quad
$(\mathcal{A},xAByACzBCt) \longleftrightarrow (\mathcal{A},xBAyCAzCBt)$
\\
where the projections of the $\alpha$-alphabets on the right hand side
of the moves H1 and H2 are restrictions of the corresponding
$\alpha$-alphabets on the left hand side. 
\par
Two nanophrases are \emph{homotopic} is there exists a finite sequence
of homotopy moves and isomorphisms which transform one into the other.
As none of the moves add or remove components, the number of
components of a nanophrase is a homotopy invariant.
\par
Note that if we change $\alpha$, $\tau$ or $S$ we may get a different
kind of homotopy.
We call the triple $(\alpha,\tau,S)$ a \emph{homotopy data triple} (note
that Turaev uses the term \emph{homotopy data} \cite{Turaev:Words}).
\par
Two homotopy data triples, $(\alpha,\tau,S)$ and
$(\alpha^\prime,\tau^\prime,S^\prime)$ are \emph{isomorphic} if there
exists a bijection $f$ from $\alpha$ to $\alpha^\prime$ such that 
$\tau^\prime \circ f$ is equivalent to $f \circ \tau$
and $(a,b,c)$ is in $S$ if and only if $(f(a),f(b),f(c))$ is in
$S^\prime$.
Such an isomorphism induces a map from $\mathcal{P}(\alpha)$ to
$\mathcal{P}(\alpha^\prime)$ which is homotopy preserving
\cite{Turaev:KnotsAndWords}.
Thus the homotopies given by isomorphic homotopy data triples are
equivalent.
\par
Let $p$ be a nanophrase.
We write $\nc(p)$ for the number of components in $p$ and define $c_i(p)$
to be the $i$th component of $p$.
The \emph{rank} of $p$ is the number of distinct letters appearing in $p$.
We write it $\rank(p)$.
The \emph{homotopy rank} of $p$, which is written $\hr(p)$, is the
minimum rank attained by a nanophrase in the homotopy class of $p$.
We say that a nanophrase $p$ is \emph{minimal} if $\rank(p)$ is equal to
$\hr(p)$. 
\par
A nanoword is \emph{contractible} if it is homotopic to the empty
word $\trivial$.
Thus, for a contractible nanoword $w$, $\hr(w)$ is $0$.
\section{Some lemmas on nanophrases}\label{sec:nanophrases}
We define a map $\chi$ from $\mathcal{P}(\alpha)$ to
$\mathcal{N}(\alpha)$ as follows.
For a nanophrase $p$ in $\mathcal{P}(\alpha)$ we define $\chi(p)$ to be
the nanoword obtained by concatenating the components of $p$.
That is $w_1|w_2|\dotsc|w_n$ is mapped to $w_1w_2 \dotsc w_n$.
By definition $\chi(\zerocomp)$ is $\trivial$.
We call $\chi$ the \emph{concatenating map}.
The following lemma is Lemma 4.3 of \cite{Fukunaga:nanophrases}.
\begin{lem}[Fukunaga]\label{lem:concatenation}
Let $p_1$ and $p_2$ be nanophrases in $\mathcal{P}(\alpha)$.
If $p_1 \sim p_2$, then $\chi(p_1) \sim \chi(p_2)$.
\end{lem}
Note that in \cite{Fukunaga:nanophrases} the lemma above was stated in
the case where $S$ is diagonal.
However, it trivially extends to the general case.
\par
Let $O$ be a subset of $\hat{n}$.
Then we define $\mathcal{P}_n(\alpha,O)$ to be the subset of
$\mathcal{P}_n(\alpha)$ 
where $p$ is in $\mathcal{P}_n(\alpha,O)$ if $c_i(p)$ is $\trivial$ for
all $i$ in $O$.
Let $\sim_O$ be the equivalence relation generated by isomorphisms and
by homotopy moves which relate elements in $\mathcal{P}_n(\alpha,O)$.
Clearly $p \sim_O p^\prime$ implies $p \sim p^\prime$.
\par
We define a map $f_O$ from $\mathcal{P}_n(\alpha)$ to
 $\mathcal{P}_n(\alpha,O)$ by saying that $f_O(p)$ is the nanophrase
 derived from $p$ by deleting all letters that appear at least once
 in any component with index in $O$.
\begin{lem}\label{lem:fpreserveshomotopy}
Let $p$ and $p^\prime$ be two nanophrases in $\mathcal{P}_n(\alpha)$.
If $p \sim p^\prime$ then $f_O(p) \sim_O f_O(p^\prime)$.
\end{lem}
\begin{proof}
It is enough to check that if $p$ and $p^\prime$ are related by a single
 isomorphism or homotopy move then $f_O(p) \sim_O f_O(p^\prime)$.
\par
If $p$ and $p^\prime$ are isomorphic then $f_O(p)$ and $f_O(p^\prime)$ are
 also isomorphic.
\par
Suppose that $p$ and $p^\prime$ are related by an H1 move.
We assume that the move removes a letter $A$ from $p$.
If $A$ appeared in a component with index in $O$, then $A$ will not
 appear in $f_O(p)$ and so $f_O(p)$ is $f_O(p^\prime)$.
Otherwise, $f_O(p)$ and $f_O(p^\prime)$ are related by an H1 move.
\par
Now suppose that $p$ and $p^\prime$ are related by an H2 move.
We assume that the move removes letters $A$ and $B$ from $p$.
If $A$ appeared in a component with index in $O$, then $B$ also appears
 in such a component and so $A$ and $B$ will both be deleted in $f_O(p)$.
Thus $f_O(p)$ is $f_O(p^\prime)$.
If $A$ does not appear in a component with index in $O$, then neither
 does $B$ and $f_O(p)$ and $f_O(p^\prime)$ are related by an H2 move.
\par
Finally suppose that $p$ and $p^\prime$ are related by an H3 move.
The move affects three letters.
Let $k$ be the number of letters affected by the move which appear in a
 component with index in $O$.
Note that if one of the letters affected by the move appears in a
 component with index in $O$, then at least one of the other two letters
 must also appear in the same component.
Thus $k$ must be $0$, $2$ or $3$.
If $k$ is $2$ or $3$ then $f_O(p)$ is $f_O(p^\prime)$.
If $k$ is $0$ then $f_O(p)$ and $f_O(p^\prime)$ are related by an H3
 move.
\end{proof}
\begin{lem}\label{lem:embedding}
Let $p$ and $p^\prime$ be two nanophrases in $\mathcal{P}_n(\alpha,O)$.
If $p \sim p^\prime$ then $p \sim_O p^\prime$.
\end{lem}
\begin{proof}
If $p \sim p^\prime$ then there exists a sequence of nanophrases
 in $\mathcal{P}_n(\alpha)$, where each consecutive pair in the sequence
 is related by an isomorphism or a single homotopy move.
By applying $f_O$ to each nanophrase in the sequence, we get a sequence of
 nanophrases in $\mathcal{P}_n(\alpha,O)$ showing $p \sim_O p^\prime$. 
\end{proof}
\begin{lem}\label{lem:simplify}
Let $p$ and $p^\prime$ be two nanophrases in $\mathcal{P}_n(\alpha)$ and
 let $O$ be a subset of the set of indices of empty components in
 $p^\prime$. 
If $p \sim p^\prime$, then  $p \sim f_O(p)$.
\end{lem}
\begin{proof}
We have
\begin{equation*}
f_O(p) \sim_O f_O(p^\prime) = p^\prime \sim p,
\end{equation*}
where the first relation is given by Lemma~\ref{lem:fpreserveshomotopy},
 the second relation follows from the definition of $f_O$ and the third
 relation is the assumption. 
\end{proof}
\begin{lem}\label{lem:hr-emptycomponent}
Let $p$ be a nanophrase in $\mathcal{P}_n(\alpha)$ such that 
$\rank(p) = \hr(p)$.
Let $p^\prime$ be a nanophrase in $\mathcal{P}_n(\alpha)$ such that the
 $i$th component of $p^\prime$ is empty.
If $p \sim p^\prime$, then the $i$th component of $p$ is empty.
\end{lem}
\begin{proof}
Suppose the $i$th component of $p$ is not empty.
Let $O$ be the set $\{i\}$.
Then by Lemma~\ref{lem:simplify}, $p \sim f_O(p)$.
Now $\rank(f_O(p))$ is less than $\rank(p)$, as the $i$th component of $p$ is
 not empty.
Then we have
\begin{equation*}
\hr(p) \leq \rank(f_O(p)) < \rank(p) = \hr(p),
\end{equation*}
which is a contradiction.
Thus the $i$th component of $p$ must be empty.
\end{proof}
We write $O(p)$ for the set of indices of components that are empty in
$p$.
\begin{lem}
Let $p$ and $p^\prime$ be two nanophrases in $\mathcal{P}_n(\alpha)$.
If $p \sim p^\prime$ and $\rank(p) = \rank(p^\prime) = \hr(p)$ then
 $O(p) = O(p^\prime)$.  
\end{lem}
\begin{proof}
For any element $i$ in $O(p^\prime)$, Lemma~\ref{lem:hr-emptycomponent}
 implies $i$ is in $O(p)$.
In other words $O(p^\prime) \subseteq O(p)$.
By symmetry, we get the opposite inclusion and so $O(p) = O(p^\prime)$.
\end{proof} 
Let $p$ be an $n$-component nanophrase and let $O$ be a subset of
$\hat{n}$ of size $\abs{O}$.
Define $x(p,O)$ to be the $n - \abs{O}$ component nanophrase derived
from $p$ by deleting all 
components with index appearing in $O$ and all letters which appear in
those components.
\begin{lem}\label{lem:subphraseinvariance}
Let $p$ and $p^\prime$ be two nanophrases in $\mathcal{P}_n(\alpha)$.
If $p \sim p^\prime$ then $x(p,O) \sim x(p^\prime,O)$. 
\end{lem}
\begin{proof}
Let $p_1$ and $p_2$ be two nanophrases in $\mathcal{P}_n(\alpha,O)$.
If $p_1 \sim_O p_2$ then $x(p_1,O) \sim x(p_2,O)$.
The result then follows from Lemma~\ref{lem:fpreserveshomotopy}.
\end{proof}
Lemma \ref{lem:subphraseinvariance} corresponds to the well-known fact
that a sub-link is an invariant of a link.
\section{Factorizations of homotopy data triples}\label{sec:factoring-triples}
Let $(\alpha,\tau,S)$ be a homotopy data triple.
From now on, we will assume that $\alpha$ is finite.
Let $\beta$ be a $\tau$-invariant subset of $\alpha$ and $\gamma$ be the
set $\alpha - \beta$.
Note that $\gamma$ is also $\tau$-invariant.
We define $\tau_\beta$ to be the restriction of $\tau$ to $\beta$ and
$\tau_\gamma$ to be the restriction of $\tau$ to $\gamma$.
We define $S_\beta$ to be the intersection of $S$ and
$\beta \times \beta \times \beta$ and $S_\gamma$ to be the intersection
of $S$ and $\gamma \times \gamma \times \gamma$.
If every element of $S$ appears either in $S_\beta$ or $S_\gamma$ (that
is, $S$ is equal to $S_\beta \cup S_\gamma$), we say that
$(\beta,\tau_\beta,S_\beta)$ is a \emph{factor} of $(\alpha,\tau,S)$.
Note that if $(\beta,\tau_\beta,S_\beta)$ is a factor of
$(\alpha,\tau,S)$, so is $(\gamma,\tau_\gamma,S_\gamma)$.
\par
Any homotopy data triple $(\alpha,\tau,S)$ is a factor of itself, as
$\alpha$ is a $\tau$-invariant subset of $\alpha$ and $S$ is contained
in $\alpha \times \alpha \times \alpha$.
The homotopy data triple $(\trivial,\tau_\trivial,\trivial)$, where
$\tau_\trivial$ is the empty map, is also a factor of any homotopy data
triple.
We note however that $(\trivial,\tau_\trivial,\trivial)$ does not give a
homotopy because the concept of an $\alpha$-alphabet is not defined when
$\alpha$ is $\trivial$.
\par
A factor $(\beta,\tau_\beta,S_\beta)$ of a homotopy data triple
$(\alpha,\tau,S)$ is said to be a \emph{proper factor} of
$(\alpha,\tau,S)$ if $\beta$ is a non-empty proper subset of $\alpha$.
We say that a homotopy data triple $(\alpha,\tau,S)$ is \emph{composite}
if it has a proper factor.
If a homotopy data triple $(\alpha,\tau,S)$ has no proper factors and
$\alpha$ is non-empty, then we say that the homotopy data triple is
\emph{prime}.
\begin{ex}\label{ex:compositehdt}
Let $\alpha$ be the set $\{a,b,c,d\}$. 
Let $\tau$ be the map which swaps $a$ with $b$ and $c$ with $d$. 
Let $S$ be the set $\{(a,b,a),(c,d,c)\}$.
Now $\alpha$ has two orbits under $\tau$, so if $(\alpha,\tau,S)$ has a
 proper factor $(\beta,\tau_\beta,S_\beta)$, 
 then $\beta$ must contain one of those orbits.
Without loss of generality, we assume $\beta$ is the set $\{a,b\}$.
Set $\gamma$ to be the set $\{c,d\}$.
Then $\tau_\beta$, the restriction of $\tau$ to $\beta$, is the map which
 swaps $a$ with $b$.
The restriction of $\tau$ to $\gamma$, $\tau_\gamma$, is the map that
 swaps $c$ with $d$.
The set $S_\beta$, the intersection of $S$ and
$\beta \times \beta \times \beta$, is $\{(a,b,a)\}$.
The set $S_\gamma$, the intersection of $S$ and
$\gamma \times \gamma \times \gamma$, is $\{(c,d,c)\}$.
Now $S$ is equal to $S_\beta \cup S_\gamma$, so
 $(\beta,\tau_\beta,S_\beta)$ is a proper factor of $(\alpha,\tau,S)$.
Thus $(\alpha,\tau,S)$ is composite.
\end{ex}
\begin{ex}
Let $\alpha$ and $\tau$ be the same as in
 Example~\ref{ex:compositehdt} and let $S$ be the set
 $\{(a,b,c),(b,c,d)\}$.
As before, $\alpha$ has two orbits under $\tau$, so if $(\alpha,\tau,S)$ has a
 proper factor $(\beta,\tau_\beta,S_\beta)$, 
 then $\beta$ must contain one of those orbits.
We assume $\beta$ is the set $\{a,b\}$.
Then $\gamma$, $\tau_\beta$ and $\tau_\gamma$ are the same as in
 Example~\ref{ex:compositehdt}.
However, in this case both $S_\beta$ and $S_\gamma$ are empty.
Clearly $S$ is not equal to $S_\beta \cup S_\gamma$, so
 $(\beta,\tau_\beta,S_\beta)$ is not a factor of $(\alpha,\tau,S)$.
Thus $(\alpha,\tau,S)$ is prime.
\end{ex}
We remark that if $\alpha$ has only one orbit under $\tau$ then the
homotopy data triple $(\alpha,\tau,S)$ must be prime.
\par
In the following example we consider the case where $S$ is diagonal.
\begin{ex}\label{ex:diagonalshdt}
Let $\alpha_G$ be the set $\{a\}$, let $\tau_G$ be the identity map and
 let $S_G$ be the set $\{(a,a,a)\}$.
Let $\alpha_F$ be the set $\{a,b\}$, let $\tau_F$ be the map swapping
 $a$ and $b$, and let $S_F$ be the set $\{(a,a,a),(b,b,b)\}$.
Then both $(\alpha_G,\tau_G,S_G)$ and $(\alpha_F,\tau_F,S_F)$ are
 prime.
We note that $(\alpha_G,\tau_G,S_G)$ gives the open Gauss phrase
 homotopy and the homotopy data triple $(\alpha_F,\tau_F,S_F)$ gives the
 open flat virtual link homotopy (see, for example,
 \cite{Gibson:gauss-phrase} or \cite{Turaev:KnotsAndWords}).
\par
Consider the general case of a homotopy data triple $(\alpha,\tau,S)$
 where $S$ is diagonal.
It is easy to check that if $(\alpha,\tau,S)$ is prime, then it is
 isomorphic to $(\alpha_G,\tau_G,S_G)$ or $(\alpha_F,\tau_F,S_F)$.
Otherwise, $(\alpha,\tau,S)$ is composite and has a proper factor which
 is isomorphic to either $(\alpha_G,\tau_G,S_G)$ or
 $(\alpha_F,\tau_F,S_F)$.
\end{ex}
We now define the product of two homotopy data triples.
Let $(\alpha_1,\tau_1,S_1)$ and $(\alpha_2,\tau_2,S_2)$ be homotopy data
triples.
Let $\alpha$ be the subset of $\{1,2\} \times (\alpha_1 \cup \alpha_2)$
given by
\begin{equation*}
\{(1,a)\;|\;a\in\alpha_1\} \cup \{(2,a)\;|\;a\in\alpha_2\}.
\end{equation*}
We define $\tau$ to be the involution on $\alpha$ where $(i,a)$ maps to
$(i,\tau_i(a))$ for each $(i,a)$ in $\alpha$.
We define $S$ to be the subset of $\alpha \times \alpha \times \alpha$
given by
\begin{equation*}
\{((1,a),(1,b),(1,c))\;|\;(a,b,c)\in S_1\} \cup 
\{((2,a),(2,b),(2,c))\;|\;(a,b,c)\in S_2\}.
\end{equation*}
Then the product of $(\alpha_1,\tau_1,S_1)$ and $(\alpha_2,\tau_2,S_2)$
is $(\alpha,\tau,S)$.
\par
We note that if $\alpha_1$ and $\alpha_2$ are disjoint, then the product
of $(\alpha_1,\tau_1,S_1)$ and $(\alpha_2,\tau_2,S_2)$ is isomorphic to
$(\alpha_1 \cup \alpha_2, \tau, S_1 \cup S_2)$ where $\tau(a)$ is
defined to be $\tau_1(a)$ if $a$ is in $\alpha_1$ and $\tau_2(a)$ if $a$
is in $\alpha_2$.
\par
The product of homotopy data triples induces a product on the
isomorphism classes of homotopy data triples.
This product is well-defined because, if $(\alpha_1,\tau_1,S_1)$
is isomorphic to $(\alpha^\prime_1,\tau^\prime_1,S^\prime_1)$ and
$(\alpha_2,\tau_2,S_2)$ is isomorphic to
$(\alpha^\prime_2,\tau^\prime_2,S^\prime_2)$, then the product of
$(\alpha_1,\tau_1,S_1)$ and $(\alpha_2,\tau_2,S_2)$ is isomorphic to the
product of $(\alpha^\prime_1,\tau^\prime_1,S^\prime_1)$ and
$(\alpha^\prime_2,\tau^\prime_2,S^\prime_2)$.
\par
Suppose $(\beta,\tau_\beta,S_\beta)$ is a factor of
$(\alpha,\tau,S)$.
Let $\gamma$ be the set $\alpha - \beta$,
let $\tau_\gamma$ be the restriction of $\tau$ to $\gamma$ and let
$S_\gamma$ be the intersection of $S$ and
$\gamma \times \gamma \times \gamma$.
Then the product of $(\beta,\tau_\beta,S_\beta)$ and
$(\gamma,\tau_\gamma,S_\gamma)$ is isomorphic to $(\alpha,\tau,S)$.
\par
It is easy to check that up to isomorphism, the product on homotopy data
triples is commutative and associative.
The unit of the product is the triple
$(\trivial,\tau_\trivial,\trivial)$.
\par
We now show that, up to isomorphism and reordering, any homotopy data
triple can be uniquely written as a product of prime homotopy data
triples.
We first state some lemmas.
\begin{lem}\label{lem:hdt-primefactor}
Suppose $h$ is a homotopy data triple that can be represented as the
 product of $f$ and $g$.
Suppose $k$ is a prime factor of $h$.
Then $k$ is a prime factor of $f$ or $g$.
\end{lem}
\begin{proof}
In this proof, the data for a homotopy data triple $x$ is written
 $\alpha_x$, $\tau_x$ and $S_x$.
As $h$ can be represented as the product $fg$, $\alpha_f$ and $\alpha_g$
 are $\tau_h$-invariant sets which partition $\alpha_h$.
\par
Now suppose that $\alpha_k \cap \alpha_f$ and $\alpha_k \cap \alpha_g$
 are both non-empty.
Write $\alpha_l$ for $\alpha_k \cap \alpha_f$ and $\alpha_m$ for
 $\alpha_k \cap \alpha_g$.
Then both $\alpha_l$ and $\alpha_m$ are $\tau_h$-invariant sets.
We define $S_l$ as the intersection of $S_h$ and
$\alpha_l \times \alpha_l \times \alpha_l$ and define $S_m$ as the
 intersection of $S_h$ and
$\alpha_m \times \alpha_m \times \alpha_m$.
Since $S_h$ is equal to $S_f \cup S_g$, $S_k$ must equal 
$S_l \cup S_m$.
This implies that $(\alpha_l,\tau_l,S_l)$ is a factor of $k$ where
 $\tau_l$ is the restriction of $\tau_h$ to $\alpha_l$.
However, this contradicts the primality of $k$.
Thus $\alpha_k$ is wholly contained in either $\alpha_f$ or $\alpha_g$.
\par
Without loss of generality, we assume that $\alpha_k$ is a subset of
 $\alpha_f$.
Let $\alpha_p$ be the set $\alpha_f - \alpha_k$ and let $S_p$ be 
the intersection of $S_f$ and
$\alpha_p \times \alpha_p \times \alpha_p$.
Let $\alpha_q$ be the set $\alpha_h - \alpha_k$ and let $S_q$ be 
the intersection of $S_h$ and
$\alpha_q \times \alpha_q \times \alpha_q$.
Now suppose that $S_k \cup S_p$ is not equal to $S_f$.
Then this would imply that $S_k \cup S_q$ is not equal to $S_h$.
However, this contradicts the fact that $k$ is a factor of $h$.
Thus $S_k \cup S_p$ is equal to $S_f$ and $k$ is a factor of $f$.
\end{proof}
\begin{lem}\label{lem:hdt-isomorphicfactors}
Let $g$, $h$ be homotopy data triples and $k$ be a prime homotopy data
 triple.
If the product of $g$ and $k$ is isomorphic to the product of $h$ and
 $k$, then $g$ and $h$ are isomorphic.
\end{lem}
\begin{proof}
We write $gk$ for the product of $g$ and $k$ and $hk$ for the product of
 $h$ and $k$.
Let $f$ be the isomorphism from $gk$ to $hk$.
As $k$ is a prime factor of $gk$, the image of $k$, $f(k)$, is a prime
 factor of $hk$.
Thus by Lemma~\ref{lem:hdt-primefactor}, $f(k)$ is either a factor of
 $h$ or of $k$.
\par
If $f(k)$ is a factor of $k$, then $f(k)$ is $k$ as they have the same
 size.
Thus the image of $g$ must be $h$ and $f$ restricted to $g$ gives an
 isomorphism from $g$ to $h$.
\par
If $f(k)$ is a factor of $h$, then $h$ can be written $kh^\prime$ for
 some $h^\prime$.
Then $f$ maps $g$ onto $h^\prime k$ and this gives an isomorphism from
 $g$ to $h^\prime k$.
Since the product of homotopy data triples is commutative up to
 isomorphism, $h^\prime k$ is isomorphic to $kh^\prime$.
Thus $g$ is isomorphic to $h$.
\end{proof}
\begin{lem}\label{lem:hdt-findfactor}
Let $h$ be a prime homotopy data triple and $h_i$ be a set of prime
 homotopy data triples, for $i$ running from $1$ to $n$.
Let $k$ be the product of the $h_i$.
If $h$ is a factor of $k$, then $h$ is isomorphic to one of the $h_i$.
\end{lem}
\begin{proof}
We prove by induction on $n$.
If $n$ is $1$ then $k$ is isomorphic to $h_1$ which is prime.
Therefore $h$ must be isomorphic to $k$.
\par
We now assume the statement is true for $n-1$ and prove it for $n$.
If $h$ is isomorphic to $h_1$ then the statement is true.
If not, Lemma~\ref{lem:hdt-primefactor} implies that $h$ must be a
 factor of the product of $h_i$ for $i$ running from $2$ to $n$.
Then this product has $n-1$ triples and so by assumption, $h$ is
 isomorphic to one of the $h_i$ for $i$ not equal to $1$.
\end{proof}
We then have the following proposition.
\begin{prop}\label{prop:uft}
Any homotopy data triple can be represented as a product of prime
 homotopy data triples.
This representation is unique up to isomorphism and the order of the
 triples in the product. 
\end{prop}
\begin{proof}
Existence of such a representation is easy to show.
If the homotopy data triple is prime we already have our representation.
If not, it is composite and can be written as the product of two
 homotopy data triples.
We consider each of these in turn and factor any that are composite.
We continue this process until only prime factors remain.
\par
Let $h$ be a homotopy data triple and suppose it can be factored
 into triples $h_i$ for $i$ running from $1$ to $m$ and also into
 triples $h^\prime_j$ for $j$ running from $1$ to $n$.
Without loss of generality we may assume that $n$ is greater than or
 equal to $m$.
\par
We prove uniqueness by induction on $m$.
For if $m$ is $1$ then $h_1$ is prime and is isomorphic to $h$.
Thus $n$ must also be $1$ and $h^\prime_1$ is isomorphic to $h_1$.
\par
We now assume the statement is true for $m-1$ and prove it for $m$.
By Lemma~\ref{lem:hdt-findfactor}, $h_1$ must be isomorphic to one of
 the $h^\prime_j$.
By reordering the $h^\prime_j$ if necessary, we may assume that it is
 $h^\prime_1$.
Then by Lemma~\ref{lem:hdt-isomorphicfactors} the product of $h_i$ with
 $i$ running from $2$ to $m$ and the product of $h^\prime_j$ for $j$
 running from $2$ to $n$ are isomorphic.
As the product of $h_i$ with $i$ running from $2$ to $m$ has $m-1$
 factors, we use the induction assumption to show that $n$ is equal to
 $m$ and, after appropriate reordering of the $h^\prime_j$, $h_j$ is
 isomorphic to $h^\prime_j$ for each $j$.
\end{proof}
\begin{ex}\label{ex:diagonalfactorization}
Let $(\alpha,\tau,S)$ be a homotopy data triple with diagonal
 $S$.
It is easy to check that $(\alpha,\tau,S)$ is isomorphic to the product
\begin{equation*}
(\alpha_G,\tau_G,S_G)^k (\alpha_F,\tau_F,S_F)^l
\end{equation*}
where $k$ and $l$ are non-negative integers such that $k+l$ is greater
 than $0$ and $(\alpha_G,\tau_G,S_G)$ and $(\alpha_F,\tau_F,S_F)$ are
 the prime homotopy data triples defined in
 Example~\ref{ex:diagonalshdt}.
\end{ex}
The following lemma extends Turaev's Lemma~3.3.1 in
\cite{Turaev:Words} to the case of general $S$.
\begin{lem}\label{lem:subhomotopyinvariance}
Let $(\alpha,\tau_\alpha,S_\alpha)$ be a homotopy data triple and let 
$(\beta,\tau_\beta,S_\beta)$ be a factor of
 $(\alpha,\tau_\alpha,S_\alpha)$.
We denote the equivalence relation given by
 $(\alpha,\tau_\alpha,S_\alpha)$ as $\sim_\alpha$ and the equivalence
 relation given by $(\beta,\tau_\beta,S_\beta)$ as $\sim_\beta$.
Now let $p$ and $p^\prime$ be two nanophrases in $\mathcal{P}(\beta)$.
If $p \sim_\alpha p^\prime$, then $p \sim_\beta p^\prime$.
\end{lem}
\begin{proof}
This is proved in the same way as Lemma~3.3.1 of
\cite{Turaev:Words}. 
\end{proof}
\begin{rem}
In Lemma~3.3.1 of \cite{Turaev:Words}, Turaev proved the case
 where $S_\alpha$ is diagonal and $p$ and $p^\prime$ are \'etale
 words.
An \'etale word is a word on an $\alpha$-alphabet, so all nanowords are
 \'etale words (see \cite{Turaev:Words} for more details).
\end{rem}
\section{An invariant for nanowords}\label{sec:nanoword-inv}
In this section we fix a composite homotopy data triple
$(\alpha,\tau,S)$ and let its prime factors be denoted by
$(\alpha_i,\tau_i,S_i)$ for $i$ 
running from $1$ to $k$ for some $k$ greater than $1$.
Let $\sim$ be the homotopy given by $(\alpha,\tau,S)$ and, for each $i$,
let $\sim_i$ be the homotopy given by $(\alpha_i,\tau_i,S_i)$.
\par
For any positive integer $n$, we say that a map from $\hat{n}$ to
$\hat{k}$ is \emph{locally variable} if there does not exist an $i$ for which 
$\theta(i)$ equals $\theta(i+1)$.
Let $\mathcal{P}_A(\alpha)$ be the set of pairs $(p,\theta)$,
where $p$ is a nanophrase over $\alpha$
and $\theta$ is a locally variable map from $\widehat{\nc(p)}$ to
$\hat{k}$.
We extend the definition of isomorphism and homotopy moves of
nanophrases to $\mathcal{P}_A(\alpha)$ by saying that 
$(p,\theta) \sim (p^\prime,\theta)$ if $p \sim p^\prime$.
\par
Let $\mathcal{P}_R(\alpha)$ be the subset of $\mathcal{P}_A(\alpha)$
consisting of pairs 
$(p,\theta)$ for which every letter $X$ appearing in the $i$th component
of $p$, $\xType{X}$ is in $\alpha_{\theta(i)}$, for all $i$.
Note that if $p$ is a nanophrase over $\alpha$ with no empty components then
if $(p,\theta)$ and $(p,\theta^\prime)$ are both in
$\mathcal{P}_R(\alpha)$ we can 
conclude that $\theta$ and $\theta^\prime$ are equal.
\par
Note that we can define a map $\Gamma$ from $\mathcal{P}_A(\alpha)$ to
$\mathcal{P}_R(\alpha)$ 
as follows.
Given $(p,\theta)$ in $\mathcal{P}_A(\alpha)$ we derive $p^\prime$ from $p$ by
deleting all letters $X$ which have projections not matching $\theta$.
In other words, for each $i$, we delete every letter $X$ in the $i$th
component of $p$ for which $\xType{X}$ is not in $\alpha_{\theta(i)}$.
\par
We define an equivalence relation on $\mathcal{P}_R(\alpha)$, written
$\sim_K$, as follows.
Let $(p,\theta)$ and $(p^\prime,\theta)$ be elements of 
$\mathcal{P}_R(\alpha)$.
Then $(p,\theta) \sim_K (p^\prime,\theta)$
if there exists a sequence of elements $(p_i,\theta)$ of 
$\mathcal{P}_R(\alpha)$, for $i$ running from $0$ to 
$r$ for some $r$, such that $p_0$ is $p$, $p_r$ is $p^\prime$ and $p_i$
is related to $p_{i+1}$ by a single homotopy move or an isomorphism.
\par
Note that $(p,\theta) \sim_K (p^\prime,\theta)$ implies
$(p,\theta) \sim (p^\prime,\theta)$. 
The following lemmas show that $\sim_K$ can be viewed as a restriction
of $\sim$ to $\mathcal{P}_R(\alpha)$.
\begin{lem}
Let $(p,\theta)$ and $(p^\prime,\theta)$ be two elements in
 $\mathcal{P}_A(\alpha)$. 
If $(p,\theta) \sim (p^\prime,\theta)$ then 
$\Gamma((p,\theta)) \sim_K \Gamma((p^\prime,\theta))$. 
\end{lem}
\begin{proof}
This can be proved in a similar way to
 Lemma~\ref{lem:fpreserveshomotopy}.
\end{proof}  
\begin{lem}\label{lem:restrictionembedding}
Let $(p,\theta)$ and $(p^\prime,\theta)$ be two elements in
 $\mathcal{P}_R(\alpha)$. 
If $(p,\theta) \sim (p^\prime,\theta)$ then 
$(p,\theta) \sim_K (p^\prime,\theta)$. 
\end{lem}
\begin{proof}
This can be proved in a similar way to
 Lemma~\ref{lem:embedding}.
\end{proof}
Let $(p,\theta)$ be an element of $\mathcal{P}_R(\alpha)$.
We fix an integer $i$ and let $O_i$ be the subset of $\widehat{\nc(p)}$
consisting of elements $j$ where $\theta(j)$ is equal to $i$.
Then we define $s_i((p,\theta))$ to be $x(p,O_i)$ (which was defined in
Section~\ref{sec:nanophrases}). 
Note that $s_i((p,\theta))$ is the subnanophrase of $p$
consisting of the components containing letters which map to
$\alpha_i$ and so $s_i((p,\theta))$ is a nanophrase over $\alpha_i$.
\begin{prop}\label{prop:factorequivalence}
Let $(p,\theta)$ and $(p^\prime,\theta)$ be two elements of
 $\mathcal{P}_R(\alpha)$.
Then $(p,\theta) \sim_K (p^\prime,\theta)$ if and only if
$s_i((p,\theta)) \sim_i s_i((p^\prime,\theta))$ for all $i$. 
\end{prop}
\begin{proof}
We first assume that 
$s_i((p,\theta)) \sim_i s_i((p^\prime,\theta))$ for all $i$. 
Note that each component of $p$ appears as a component in exactly one of
 the nanophrases $s_i((p,\theta))$.
Since $s_1((p,\theta)) \sim_1 s_1((p^\prime,\theta))$, there is a
 sequence of homotopy moves going from $s_1((p,\theta))$ to
 $s_1((p^\prime,\theta))$ which can be applied directly to the
 corresponding components in $(p,\theta)$.
Call the result $(p_1,\theta)$.
Then $s_1(p_1,\theta)$ is isomorphic to $s_1((p^\prime,\theta))$ and
 $s_j(p_1,\theta)$ is equal to $s_j(p,\theta)$ for all $j$ not equal to
 $1$.
We also have $(p_1,\theta) \sim_K (p,\theta)$.
Starting with $(p_1,\theta)$ we repeat the process, this time using the
 fact that $s_2((p,\theta)) \sim_2 s_2((p^\prime,\theta))$.
The result is $(p_2,\theta)$ which is $\sim_K$ equivalent to
 $(p,\theta)$.
By repeating the process for each $i$, the final result $(p_k,\theta)$
 is $\sim_K$ equivalent to $(p,\theta)$ and isomorphic to
 $(p^\prime,\theta)$.
Thus we have shown that $(p,\theta) \sim_K (p^\prime,\theta)$.
\par
We now assume that $(p,\theta) \sim_K (p^\prime,\theta)$.
This implies that $p \sim p^\prime$.
Now, fixing an integer $i$, let $O_i$ be the subset of $\widehat{\nc(p)}$
consisting of elements $j$ where $\theta(j)$ is equal to $i$.
Then, by definition, $s_i((p,\theta))$ is $x(p,O_i)$ and
 $s_i((p^\prime,\theta))$ is $x(p^\prime,O_i)$. 
Since $p \sim p^\prime$, Lemma~\ref{lem:subphraseinvariance} implies
$x(p,O_i) \sim x(p^\prime,O_i)$ and thus, using
 Lemma~\ref{lem:subhomotopyinvariance},
$x(p,O_i) \sim_i x(p^\prime,O_i)$.
\end{proof}
We define some new moves on elements of $\mathcal{P}_R(\alpha)$ which allow us to
remove or add empty components.
Let $(p,\theta)$ be an element of $\mathcal{P}_R(\alpha)$.
We write $n$ for $\nc(p)$ and $w_j$ for the $j$th component of $p$.
If $w_i$ is $\trivial$ we can apply a simple reduction or concatenating
reduction which are defined as follows.
\par
Simple Reduction: if $i=1$, $i=n$ or $\theta(i-1) \neq \theta(i+1)$,
\begin{equation*}
(w_1|\dotsc|w_{i-1}|\trivial|w_{i+1}|\dotsc|w_n,\theta) \longrightarrow 
(w_1|\dotsc|w_{i-1}|w_{i+1}|\dotsc|w_n,\theta^\prime)
\end{equation*} 
where $\theta^\prime$ is a map from $\widehat{n-1}$ to $\hat{k}$ defined by
\begin{equation*}
\theta^\prime(j) = 
\begin{cases}
\theta(j) & \text{if } j < i \\
\theta(j+1) & \text{if } j \geq i.
\end{cases} 
\end{equation*}
Note that if $n$ is $1$ and $w_1$ is $\trivial$, applying a simple
reduction gives the pair $(\zerocomp,\theta_\emptyset)$ where
$\theta_\emptyset$ is the empty map (from $\emptyset$ to $\hat{k}$). 
The inverse of a simple reduction is called a \emph{simple augmentation}.
It allows us to insert a new empty component into $p$ either between two
components or at one of the ends, as long as we can make a corresponding
change to $\theta$ to get a locally variable map.
\par
Concatenating Reduction: if $i \neq 1$, $i \neq n$ and 
$\theta(i-1) = \theta(i+1)$,
\begin{equation*}
(w_1|\dotsc|w_{i-1}|\trivial|w_{i+1}|\dotsc|w_n,\theta) \longrightarrow 
(w_1|\dotsc|w_{i-1}w_{i+1}|\dotsc|w_n,\theta^\prime)
\end{equation*}
where $\theta^\prime$ is a map from $\widehat{n-2}$ to $\hat{k}$ defined by
\begin{equation*}
\theta^\prime(j) = 
\begin{cases}
\theta(j) & \text{if } j < i \\
\theta(j+2) & \text{if } j \geq i.
\end{cases} 
\end{equation*}
The inverse of this move is called a \emph{splitting augmentation}.
It allows us to split a component of $p$ into two parts (where possibly
one or both may be empty),
as long as we can make a corresponding change to $\theta$ to get a
locally variable map.
\par
Collectively we refer to simple reductions and concatenating reductions
as \emph{reductions} and simple augmentations and splitting
augmentations as \emph{augmentations}.
\begin{rem}
Suppose $(p,\theta)$ is an element of $\mathcal{P}_R(\alpha)$ such that the
 $i$th component of $p$ is empty.
Then by the conditions on $i$ and $\theta$ given in the definitions of
 the reduction moves, there is one and only one way that we can apply a
 reduction to $(p,\theta)$ to remove the $i$th component.
Thus we can unambiguously refer to this reduction as \emph{the reduction
 of the $i$th component} of $(p,\theta)$.
\end{rem}
Let $\sim_P$ be the equivalence relation on $\mathcal{P}_R(\alpha)$ generated by
$\sim_K$, reductions and augmentations.
We can define a map $\Omega$ from $\mathcal{P}_R(\alpha)$ to $\mathcal{N}(\alpha)$ by
mapping $(p,\theta)$ to $\chi(p)$ (recall that $\chi$ is the
concatenating map defined in Section~\ref{sec:nanophrases}).
Then we have the following lemma.
\begin{lem}\label{lem:hinvariance}
Let $(p,\theta)$ and $(p^\prime,\theta^\prime)$ be elements of
 $\mathcal{P}_R(\alpha)$.
If $(p,\theta) \sim_P (p^\prime,\theta^\prime)$ then 
$\Omega((p,\theta)) \sim \Omega((p^\prime,\theta^\prime))$.
\end{lem}
\begin{proof}
If $(p,\theta) \sim_K (p^\prime,\theta^\prime)$, then
$\theta$ and $\theta^\prime$ are equal and $p \sim p^\prime$.
Then by Lemma~\ref{lem:concatenation}, $\chi(p) \sim \chi(p^\prime)$ and
 so $\Omega((p,\theta)) \sim \Omega((p^\prime,\theta^\prime))$.
So, to prove the result, it is enough to prove the case where
 $(p^\prime,\theta^\prime)$ is derived from $(p,\theta)$ by a
 single reduction.
However in this case, it is easy to see that $\chi(p)$ and
 $\chi(p^\prime)$ must be equal and so
$\Omega((p,\theta)) \sim \Omega((p^\prime,\theta^\prime))$.
\end{proof}
We define a map $\psi$ from $\mathcal{N}(\alpha)$ to $\mathcal{P}_R(\alpha)$ as
follows.
Let $w$ be a nanoword in $\mathcal{N}(\alpha)$.
If $w$ is $\trivial$ then we map $w$ to $\zerocomp$.
Otherwise, there is a unique integer $r > 0$, a unique sequence of words
$w_1, w_2, \dotsc, w_r$ and a unique map
$\theta$ from $\hat{r}$ to $\hat{k}$ such that
\begin{enumerate}
\item
each $w_i$ is not $\trivial$;
\item
$w$ is equal to $w_1w_2 \dotso w_r$;
\item
for all $i$, if $X$ appears in $w_i$, $\xType{X}$ is in
     $\alpha_{\theta(i)}$;
\item
for any $i$ from $1$ to $r-1$, $\theta(i)$ is not equal to
     $\theta(i+1)$.
\end{enumerate}
Define $p$ to be the nanophrase $w_1|w_2|\dotsc|w_r$ where the letters
in $p$ have the same projection as they do in $w$.
Then $(p,\theta)$ is in $\mathcal{P}_R(\alpha)$.
We define $\psi(w)$ to be $(p,\theta)$.
\begin{ex}
Let $\alpha$ be the set $\{a,b\}$.
Suppose that under factorization it factors into two sets, $\{a\}$, which
 we call $\alpha_1$, and $\{b\}$, which we call $\alpha_2$.
Then $k$ is $2$.
\par
Let $\mathcal{A}$ be the $\alpha$-alphabet $\{A,B,C,D,E\}$ with
 projection given by $\xType{A} = \xType{B} = \xType{C} = a$ and
 $\xType{D} = \xType{E} = b$.
Now consider the nanoword $(\mathcal{A},w)$ where $w$ is $ABCBDCAEDE$.
\par
Then $w_1$ is $ABCB$, $w_2$ is $D$, $w_3$ is $CA$ and $w_4$ is $EDE$.
So $n$ is $4$ and $\theta$ is a map from $\hat{4}$ to $\hat{2}$ which
 maps $1$ and $3$ to $1$ and maps $2$ and $4$ to $2$.
Thus $\psi$ maps $(\mathcal{A},w)$ to 
$((\mathcal{A},w_1|w_2|w_3|w_4),\theta)$.
\end{ex}
Note that the image of $\psi$ consists of all elements $(p,\theta)$ in
$\mathcal{P}_R(\alpha)$ such that $p$ has no empty components.
\begin{lem}\label{lem:psiinvariance}
Let $w$ and $w^\prime$ be elements of $\mathcal{N}(\alpha)$.
If $w \sim w^\prime$, then $\psi(w) \sim_P \psi(w^\prime)$.
\end{lem}
\begin{proof}
It is enough to prove the case where $w$ and $w^\prime$ are related by
 an isomorphism or a single homotopy move.
\par
If $w$ and $w^\prime$ are related by an isomorphism, then the
 nanophrases in $\psi(w)$ and
 $\psi(w^\prime)$ are isomorphic and so
 $\psi(w) \sim_P \psi(w^\prime)$.
\par
We now assume $w$ and $w^\prime$ are related by a single homotopy move.
We write $\psi(w)$ as $(p, \theta)$ and $\psi(w^\prime)$ as
$(p^\prime, \theta^\prime)$.
\par
Suppose $w$ and $w^\prime$ are related by an H1 move.
If $\theta$ and $\theta^\prime$ are equal, then $p$ is related to
 $p^\prime$ by a corresponding H1 move.
If not, we assume that $w^\prime$ is derived from $w$ by removing a
 letter.
By making the corresponding H1 move on $(p, \theta)$ and then applying
 an appropriate reduction we can derive $(p^\prime, \theta^\prime)$.
Thus, in either case, $\psi(w) \sim_P \psi(w^\prime)$.
\par
Suppose $w$ and $w^\prime$ are related by an H2 move.
If $\theta$ and $\theta^\prime$ are equal, then $p$ is related to
 $p^\prime$ by a corresponding H2 move.
If not, we assume that $w^\prime$ is derived from $w$ by removing two
 letters.
By making the corresponding H2 move on $(p, \theta)$ and then applying
 appropriate reductions we can derive $(p^\prime, \theta^\prime)$.
Thus, in either case, $\psi(w) \sim_P \psi(w^\prime)$.
\par
If $w$ and $w^\prime$ are related by an H3 move, then $\theta$ and
 $\theta^\prime$ are equal and $p$ is related to $p^\prime$ by a
 corresponding H3 move.
Thus $\psi(w) \sim_P \psi(w^\prime)$.
\end{proof}
\begin{thm}\label{thm:bijection}
There is a bijection between the homotopy classes of
 $\mathcal{N}(\alpha)$ and the equivalence classes of $\mathcal{P}_R(\alpha)$
 under $\sim_P$.
In other words
\begin{equation*}
\mathcal{N}(\alpha) / \sim \quad \cong \quad \mathcal{P}_R(\alpha) / \sim_P.
\end{equation*} 
\end{thm}
\begin{proof}
The map $\psi$ induces a map $\psi_P$ from $\mathcal{N}(\alpha) / \sim$
 to $\mathcal{P}_R(\alpha) / \sim_P$, which by Lemma~\ref{lem:psiinvariance} is
 well-defined.
Note that for any nanoword $w$ in $\mathcal{N}(\alpha)$, $\Omega(\psi(w))$ is
 $w$.
Then by Lemma~\ref{lem:hinvariance}, $\psi_P$ is bijective.
\end{proof}
Let $\mathcal{K}(\alpha)$ be the set of equivalence classes of $\mathcal{P}_R(\alpha)$
under $\sim_K$.
Let $c$ be an element of $\mathcal{K}(\alpha)$ and $(p,\theta)$ be an element in
$c$. 
We define $\nc(c)$ to be $\nc(p)$.
Note that if $(q,\theta)$ is another element in $c$ then $\nc(p)$ is equal
to $\nc(q)$, and so $\nc(c)$ is well-defined.
\par
We say that an element $c$ of $\mathcal{K}(\alpha)$ is \emph{$i$-reducible} if
there exists a pair $(p,\theta)$ in $c$ such that the $i$th component of
$p$ is $\trivial$.
\begin{lem}\label{lem:kembedding}
Let $c$ be an element of $\mathcal{K}(\alpha)$.
Suppose $(p,\theta)$ and $(p^\prime,\theta)$ are elements in $c$
 such that $p$ and $p^\prime$ are both $i$-reducible.
Then there exists a sequence of elements of $c$ going from $(p,\theta)$
 to $(p^\prime,\theta)$ such that the nanophrases in each consecutive
 pair are related by a single homotopy move or isomorphism and each 
 nanophrase in the sequence is $i$-reducible.
\end{lem}
\begin{proof}
This is proved in a similar way to Lemma~\ref{lem:embedding}.
\end{proof}
We note that the equivalence relation $\sim_P$ on the set
$\mathcal{P}_R(\alpha)$ induces an equivalence relation on $\mathcal{K}(\alpha)$ which
we also write $\sim_P$.
\par
Let $(p^\prime,\theta^\prime)$ be the result the reduction of
the $i$th component of $(p,\theta)$ and let $c^\prime$ be the element of
$\mathcal{K}(\alpha)$ which contains $(p^\prime,\theta^\prime)$.
We say that $c^\prime$ is the \emph{$i$-reduction} of $c$.
The following lemma shows that this concept is well-defined.
\begin{lem}\label{lem:reduction-wd}
Let $c$ be an element of $\mathcal{K}(\alpha)$.
If $c^\prime$ and $c^{\prime\prime}$ are $i$-reductions of $c$,
$c^\prime$ equals $c^{\prime\prime}$. 
\end{lem} 
\begin{proof}
As $c^\prime$ is an $i$-reduction of $c$, there exists $(p,\theta)$ in
 $c$ such that the $i$th component of $p$ is $\trivial$ and the
 reduction of the $i$th component of $(p,\theta)$ gives
 $(p^\prime,\theta^\prime)$ in $c^\prime$.
Similarly, as $c^{\prime\prime}$ is an $i$-reduction of $c$, there
 exists $(q,\theta)$ in $c$ such that the $i$th component of $q$ is
 $\trivial$ and the 
 reduction of the $i$th component of $(q,\theta)$ gives
 $(q^\prime,\theta^{\prime\prime})$ in $c^{\prime\prime}$.
By Lemma~\ref{lem:kembedding} there exists a sequence of
 nanophrases $p_j$ going from $p$ to $q$ such that each $(p_j,\theta)$
 is in $c$ and the $i$th component of $p_j$ is $\trivial$.
Then for each $j$ it is easy to check that the reduction of the $i$th
 component of $(p_j,\theta)$ is equivalent, under $\sim_K$ to the
 reduction of the $i$th component of $(p_{j+1},\theta)$.
Thus $(p^\prime,\theta^\prime) \sim_K (q^\prime,\theta^{\prime\prime})$
 and so $c^\prime$ equals $c^{\prime\prime}$.
\end{proof}
We say that an element $c$ of $\mathcal{K}(\alpha)$ is \emph{reducible} if it is
$i$-reducible for some $i$.
An element $c$ of $\mathcal{K}(\alpha)$ is \emph{irreducible} if it is not
reducible.
\begin{lem}\label{lem:ijreducible}
Let $c$ be an element of $\mathcal{K}(\alpha)$ and
let $i$ and $j$ be different integers in $\widehat{\nc(c)}$.
If $c$ is both $i$-reducible and $j$-reducible, then there exists an
 element $(p,\theta)$ in $c$ such that both the $i$th and $j$th
 components of $p$ are $\trivial$.
\end{lem}
\begin{proof}
As $c$ is $i$-reducible, there exists an element $(q,\theta)$ in $c$
 such that the $i$th component of $q$ is $\trivial$.
As $c$ is $j$-reducible, there exists an element $(q^\prime,\theta)$ in
 $c$ such that the $j$th component of $q^\prime$ is $\trivial$.
Recall that for a nanophrase $p$ and a subset $O$ of
 $\widehat{\nc(p)}$, $f_O$ maps $p$ to the nanophrase derived from $p$ by
 deleting all letters that appear at least once in any component with
 index appearing in $O$. 
Let $O$ be the set $\{j\}$.
As $(q,\theta) \sim_K (q^\prime,\theta)$, by a similar argument to the
 proof of Lemma~\ref{lem:simplify}, we can show that
$(q,\theta) \sim_K (f_O(q),\theta)$.
Thus $(f_O(q),\theta)$ is in $c$ and both the $i$th and $j$th components
 of $f_O(q)$ are $\trivial$.
\end{proof}
\begin{lem}\label{lem:confluence}
Let $c$ be an element of $\mathcal{K}(\alpha)$ and
let $i$ and $j$ be integers in $\widehat{\nc(c)}$ with $i$ less
 than $j$.
Suppose $c$ is both $i$-reducible and $j$-reducible.
Let $c_i$ be the result of the $i$-reduction of $c$ and let $c_j$ be the
 result of the $j$-reduction of $c$.
Then one of the following cases holds.
\begin{enumerate}
\item\label{case:equality}
$c_i$ is equal to $c_j$;
\item\label{case:lefttriangle}
$c_i$ is a reduction of $c_j$;
\item\label{case:righttriangle}
$c_j$ is a reduction of $c_i$;
\item\label{case:diamond}
there exists an element $c^\prime$ of $\mathcal{K}(\alpha)$ which is a
     reduction both of $c_i$ and $c_j$.
\end{enumerate} 
\end{lem}
\begin{proof}
By Lemma~\ref{lem:ijreducible} we know that $c$ contains an element
 $(p,\theta)$ such that the $i$th and $j$th components of $p$ are
 $\trivial$.
We write $(p_i,\theta_i)$ for the $i$-reduction of $p$ and 
$(p_j,\theta_j)$ for the $j$-reduction of $p$.
Then $(p_i,\theta_i)$ is in $c_i$ and $(p_j,\theta_j)$ is in $c_j$.
\par
Assume that $j$ is greater than $i+2$.
Then the set of components involved in the $i$-reduction and the set of
 components involved in the $j$-reduction have no components in common.
So $(p_j,\theta_j)$ will be $i$-reducible and the $i$-reduction will be
 of the same type as the $i$-reduction of $(p,\theta)$. 
Similarly, there will be a reduction available in $(p_i,\theta_i)$
 corresponding to the $j$-reduction of $(p,\theta)$.
If the $i$-reduction of $(p,\theta)$ is simple, $(p_i,\theta_i)$ will be
 $(j-1)$-reducible.
If the $i$-reduction of $(p,\theta)$ is concatenating, $(p_i,\theta_i)$
 will be $(j-2)$-reducible.
It is easy to check that in all cases, applying the two reductions in
 either order gives the same result.
\par
Let $(p^\prime,\theta^\prime)$ be the result of applying an
 $i$-reduction to $(p_j,\theta_j)$ and let $c^\prime$ be the equivalence
 class of $(p^\prime,\theta^\prime)$ under $\sim_K$.
Then as $c^\prime$ is a reduction of both $c_i$ and $c_j$,
 Case~\ref{case:diamond} in the statement of the lemma holds.
\par
For example, the case where both reductions are simple is shown in
 Figure~\ref{fig:independent-simple}.
In the figure $w_t$ is the $t$th component of $p$ for each $t$ and
 labels on the arrows indicate which component is being reduced.
\begin{figure}[hbt]
\begin{center}
\begin{picture}(300,140)
\put(70,120){\makebox(160,20){$(w_1|\dotsc|w_{i-1}|\trivial|w_{i+1}|\dotsc|w_{j-1}|\trivial|w_{j+1}|\dotsc|w_n,\theta)$}}
\put(10,80){\makebox(140,20){$(w_1|\dotsc|w_{i-1}|w_{i+1}|\dotsc|w_{j-1}|\trivial|w_{j+1}|\dotsc|w_n,\theta_i)$}}
\put(150,40){\makebox(140,20){$(w_1|\dotsc|w_{i-1}|\trivial|w_{i+1}|\dotsc|w_{j-1}|w_{j+1}|\dotsc|w_n,\theta_j)$}}
\put(90,0){\makebox(120,20){$(w_1|\dotsc|w_{i-1}|w_{i+1}|\dotsc|w_{j-1}|w_{j+1}|\dotsc|w_n,\theta^\prime)$}}
\put(120,120){\vector(-1,-1){20}}
\put(180,120){\vector(1,-1){60}}
\put(200,40){\vector(-1,-1){20}}
\put(60,80){\vector(1,-1){60}}
\put(115,105){$i$}
\put(220,85){$j$}
\put(60,45){$j-1$}
\put(200,25){$i$}
\end{picture}
\caption{Simple reductions}
\label{fig:independent-simple}
\end{center}
\end{figure}
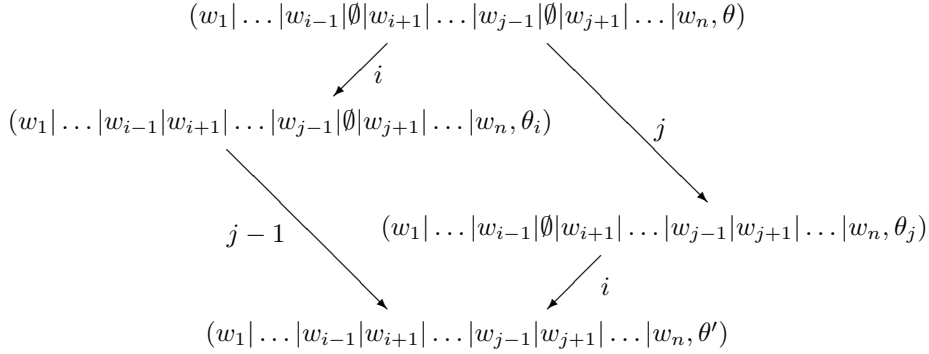
\par
Now assume that $j$ is $i+2$.
If the $i$-reduction or $j$-reduction of $(p,\theta)$ is simple then the
 situation is similar to the case where $j$ is greater than $i+2$ and so 
 Case~\ref{case:diamond} in the statement of the lemma holds.
The case where both the $i$-reduction and the $j$-reduction of
 $(p,\theta)$ are concatenating is shown in Figure~\ref{fig:merge}.
Note that because both reductions are concatenating, 
$\theta(i-1)$, $\theta(i+1)$ and $\theta(i+3)$ are all equal.
Thus, again, Case~\ref{case:diamond} in the statement of the lemma
 holds.
\begin{figure}[hbt]
\begin{center}
\begin{picture}(300,100)
\put(80,80){\makebox(140,20){$(w_1|\dotsc|w_{i-1}|\trivial|w_{i+1}|\trivial|w_{i+3}|\dotsc|w_n,\theta)$}}
\put(0,40){\makebox(120,20){$(w_1|\dotsc|w_{i-1}w_{i+1}|\trivial|w_{i+3}|\dotsc|w_n,\theta_i)$}}
\put(180,40){\makebox(120,20){$(w_1|\dotsc|w_{i-1}|\trivial|w_{i+1}w_{i+3}|\dotsc|w_n,\theta_j)$}}
\put(100,0){\makebox(100,20){$(w_1|\dotsc|w_{i-1}w_{i+1}w_{i+3}|\dotsc|w_n,\theta^\prime)$}}
\put(120,80){\vector(-1,-1){20}}
\put(180,80){\vector(1,-1){20}}
\put(200,40){\vector(-1,-1){20}}
\put(100,40){\vector(1,-1){20}}
\put(115,65){$i$}
\put(160,65){$i+2$}
\put(95,25){$i$}
\put(200,25){$i$}
\end{picture}
\caption{Concatenating reductions when $j=i+2$}
\label{fig:merge}
\end{center}
\end{figure}
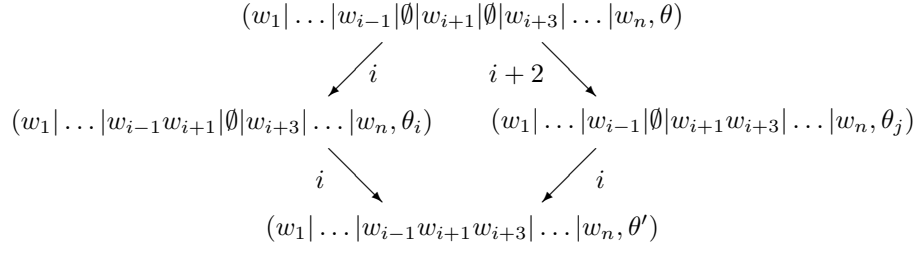
\par
Now assume that $j$ is $i+1$.
We consider cases based on the types of the reductions of $(p,\theta)$.
\par
If both the $i$-reduction and the $j$-reduction are concatenating then
 the results of the reductions are equivalent, as shown in
 Figure~\ref{fig:gap1cc}.
Thus Case~\ref{case:equality} in the statement of the lemma
 holds.
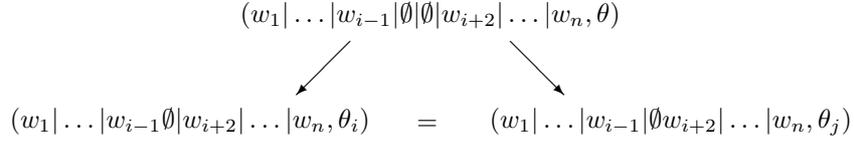
\begin{figure}[hbt]
\begin{center}
\begin{picture}(300,60)
\put(80,40){\makebox(140,20){$(w_1|\dotsc|w_{i-1}|\trivial|\trivial|w_{i+2}|\dotsc|w_n,\theta)$}}
\put(0,0){\makebox(120,20){$(w_1|\dotsc|w_{i-1}\trivial|w_{i+2}|\dotsc|w_n,\theta_i)$}}
\put(180,0){\makebox(120,20){$(w_1|\dotsc|w_{i-1}|\trivial w_{i+2}|\dotsc|w_n,\theta_j)$}}
\put(145,7){$=$}
\put(120,40){\vector(-1,-1){20}}
\put(180,40){\vector(1,-1){20}}
\end{picture}
\caption{Concatenating reductions when $j=i+1$}
\label{fig:gap1cc}
\end{center}
\end{figure}
\par
The case where the $i$-reduction is simple and the $j$-reduction is
 concatenating is shown in Figure~\ref{fig:gap1sc}.
Thus Case~\ref{case:lefttriangle} in the statement of the lemma
 holds.
\begin{figure}[hbt]
\begin{center}
\begin{picture}(300,100)
\put(80,80){\makebox(140,20){$(w_1|\dotsc|w_{i-1}|\trivial|\trivial|w_{i+2}|\dotsc|w_n,\theta)$}}
\put(0,40){\makebox(120,20){$(w_1|\dotsc|w_{i-1}|\trivial|w_{i+2}|\dotsc|w_n,\theta_i)$}}
\put(90,0){\makebox(120,20){$(w_1|\dotsc|w_{i-1}|\trivial w_{i+2}|\dotsc|w_n,\theta_j)$}}
\put(120,80){\vector(-1,-1){20}}
\put(100,40){\vector(1,-1){20}}
\put(180,80){\vector(0,-1){60}}
\put(115,65){$i$}
\put(190,45){$i+1$}
\put(95,25){$i$}
\end{picture}
\caption{Simple $i$-reduction, concatenating $j$-reduction when $j=i+1$}
\label{fig:gap1sc}
\end{center}
\end{figure}
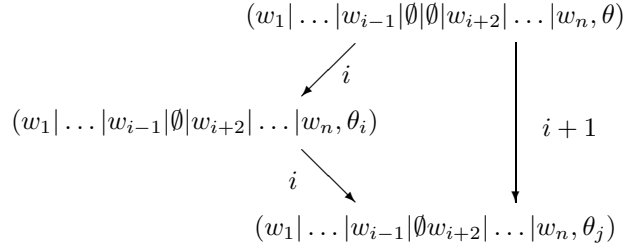
\par
The case where the $i$-reduction is concatenating and the $j$-reduction
 is simple is shown in Figure~\ref{fig:gap1cs}.
Thus Case~\ref{case:righttriangle} in the statement of the lemma
 holds.
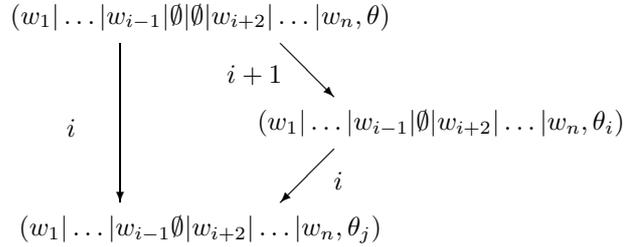
\begin{figure}[hbt]
\begin{center}
\begin{picture}(300,100)
\put(80,80){\makebox(140,20){$(w_1|\dotsc|w_{i-1}|\trivial|\trivial|w_{i+2}|\dotsc|w_n,\theta)$}}
\put(180,40){\makebox(120,20){$(w_1|\dotsc|w_{i-1}|\trivial|w_{i+2}|\dotsc|w_n,\theta_i)$}}
\put(90,0){\makebox(120,20){$(w_1|\dotsc|w_{i-1}\trivial|w_{i+2}|\dotsc|w_n,\theta_j)$}}
\put(180,80){\vector(1,-1){20}}
\put(200,40){\vector(-1,-1){20}}
\put(120,80){\vector(0,-1){60}}
\put(160,65){$i+1$}
\put(100,45){$i$}
\put(200,25){$i$}
\end{picture}
\caption{Concatenating $i$-reduction, simple $j$-reduction when $j=i+1$}
\label{fig:gap1cs}
\end{center}
\end{figure}
\par
The case where both the $i$-reduction and the $j$-reduction are simple
 splits into two subcases.
The first subcase is when $i$ is $1$ or $j$ is $n$ or $\theta(i-1)$ is
 not equal to $\theta(i+2)$.
Then $(p_i,\theta_i)$ is $i$-reducible by a simple reduction and
 $(p_j,\theta_j)$ is also $i$-reducible by a simple reduction.
Either reduction gives the same result and so Case~\ref{case:diamond} in
 the statement of the lemma holds.
\par
The second subcase is when $i$ is not $1$, $j$ is not $n$ and $\theta(i-1)$ is
 equal to $\theta(i+2)$.
Then both $(p_i,\theta_i)$ and $(p_j,\theta_j)$ are $i$-reducible by
 concatenating reductions as shown in Figure~\ref{fig:gap1ss}.
Note that in this case $p_i$ is equal to $p_j$, but as $\theta(i)$ is
 not equal to $\theta(i+1)$, $\theta_i$ and $\theta_j$ are different.
From the figure we see that Case~\ref{case:diamond} in the statement of
 the lemma holds.
\begin{figure}[hbt]
\begin{center}
\begin{picture}(300,100)
\put(80,80){\makebox(140,20){$(w_1|\dotsc|w_{i-1}|\trivial|\trivial|w_{i+2}|\dotsc|w_n,\theta)$}}
\put(0,40){\makebox(120,20){$(w_1|\dotsc|w_{i-1}|\trivial|w_{i+2}|\dotsc|w_n,\theta_i)$}}
\put(180,40){\makebox(120,20){$(w_1|\dotsc|w_{i-1}|\trivial|w_{i+2}|\dotsc|w_n,\theta_j)$}}
\put(100,0){\makebox(100,20){$(w_1|\dotsc|w_{i-1}w_{i+2}|\dotsc|w_n,\theta^\prime)$}}
\put(120,80){\vector(-1,-1){20}}
\put(180,80){\vector(1,-1){20}}
\put(200,40){\vector(-1,-1){20}}
\put(100,40){\vector(1,-1){20}}
\put(115,65){$i$}
\put(160,65){$i+1$}
\put(95,25){$i$}
\put(200,25){$i$}
\end{picture}
\caption{Simple reductions when $j=i+1$ and $\theta(i-1)=\theta(i+2)$}
\label{fig:gap1ss}
\end{center}
\end{figure}
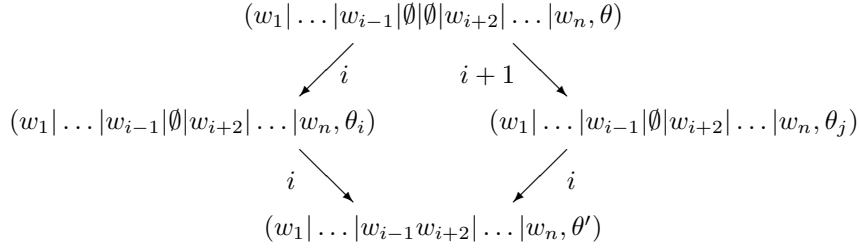
\end{proof}
Let $F$ be a set with an equivalence relation defined on it in terms of
 two types of moves, `reducing' moves and the inverses of these moves.
An element of $F$ is said to be \emph{reduced} if it does not admit a
reducing move.
A \emph{reducing chain} of elements of $F$ is a sequence 
$g_1, g_2, \dotsc, g_n$ such that $g_{i+1}$ is the result of a reducing
move applied to $g_i$ for all $i$.
\par
The following lemma, known as the Diamond Lemma was proved by Newman in
\cite{Newman:diamond}.
The version we give here is based on the version of the lemma appearing
in \cite[page~26]{Cohn:Algebra3}.
\begin{lem}[Diamond Lemma, Newman]\label{lem:diamond}
Let $F$ be a set as above, such that the following conditions are
 satisfied:
\begin{enumerate}
\item
Finiteness condition. For each element $g$ in $F$ there exists an
     integer $r$ (depending on $g$) such that no reducing chain starting
     from $g$ has more than $r$ terms.
\item
Confluence condition. If an element $g$ in $F$ can be transformed to
     $g_1$ by one reducing move and $g_2$ by another, then there exists
     an element $g^\prime$ such that, for each $i$, either $g^\prime$ is
     equal to $g_i$ or $g_i$ can be transformed into $g^\prime$ by
     applying one or more reducing moves.
\end{enumerate}
Then each equivalence class of $F$ contains exactly one reduced
 element.
\end{lem}
\begin{prop}\label{prop:reduction-wd}
Each equivalence class of $\mathcal{K}(\alpha)$ under $\sim_P$ contains exactly
 one reduced element.
\end{prop} 
\begin{proof}
We prove this using Lemma~\ref{lem:diamond}.
We just need to show that the Confluence condition and the Finiteness
 condition hold for $\mathcal{K}(\alpha)$.
\par
The fact that the Confluence condition holds is shown by
 Lemma~\ref{lem:confluence}. 
On the other hand, a reducing chain starting from $c$, an element of
 $\mathcal{K}(\alpha)$, can have no more than $\nc(c)+1$ elements, because a
 reduction move reduces the number of components by at least $1$.
Thus the Finiteness condition of Lemma~\ref{lem:diamond} also holds.
\end{proof}
We define a map $R$ from $\mathcal{K}(\alpha)$ to itself as follows.
For an element $c$ in $\mathcal{K}(\alpha)$, we define $R(c)$ to
be the reduced element in the equivalence class of $c$ under $\sim_P$.
By Proposition~\ref{prop:reduction-wd}, the map $R$ is well-defined.
\par
Assuming that we can tell whether any element of $\mathcal{K}(\alpha)$ is
reducible or not (which in general is a difficult problem), it is easy
to calculate $R(c)$ for any $c$ using an inductive process.
If $c$ is irreducible, $R(c)$ is $c$.
If not, $c$ can be reduced to some other element of $\mathcal{K}(\alpha)$, say
$c^\prime$ and $R(c)$ is equal to $R(c^\prime)$.
As was shown in the proof of Proposition~\ref{prop:reduction-wd}, the
Finiteness condition holds, so this process will terminate in a finite
number of steps.
\par
Let $w$ be a nanoword in $\mathcal{N}(\alpha)$.
Then $\psi(w)$ is an element in $\mathcal{P}_R(\alpha)$.
Let $c$ be the element of $\mathcal{K}(\alpha)$ which contains
$\psi(w)$.
Then we define $R_K(w)$ to be $R(c)$.
\begin{thm}\label{thm:invariance}
For a nanoword $w$ in $\mathcal{N}(\alpha)$, $R_K(w)$ is a homotopy
 invariant of $w$.
\end{thm}
\begin{proof}
This follows from Theorem~\ref{thm:bijection} and
 Proposition~\ref{prop:reduction-wd}.
\end{proof}
In particular, this means that $\nc(R_K(w))$ and the $\theta$
associated with $R_K(w)$ are invariants of $w$.
We write $c_R(w)$ for $\nc(R_K(w))$ and write $\theta_R(w)$ for the
$\theta$ associated with $R_K(w)$.
We also define $P_{R,i}(w)$ to be $s_i((p,\theta))$ where $(p,\theta)$
is an element in $R_K(w)$.
The nanophrase $s_i((p,\theta))$ is an nanophrase in
$\mathcal{P}(\alpha_i)$ and is, modulo $\sim_i$, an invariant of $w$.
\par
We have the following corollary of Theorem~\ref{thm:invariance} which
show that $\theta_R(w)$ and the sequence of nanophrases
$P_{R,i}(w)$ together give a complete invariant for $w$.
\begin{cor}\label{cor:completeinvariant}
Let $w$ and $w^\prime$ be nanowords over $\alpha$.
Then $w \sim w^\prime$ if and only if $\theta_R(w)$ and
 $\theta_R(w^\prime)$ are equal and, for all $i$,
$P_{R,i}(w) \sim_i P_{R,i}(w^\prime)$
\end{cor}
\begin{proof}
This follows from Theorem~\ref{thm:invariance} and
 Proposition~\ref{prop:factorequivalence}.
\end{proof} 
\par
We say that $w$ is \emph{c-minimal} if $\nc(\psi(w))$ is equal to
$c_R(w)$.
The following proposition shows how the property of being minimal and
the property of being c-minimal are related.
\begin{prop}\label{prop:reducedisminimal}
Let $w$ be a nanoword.
If $w$ is minimal, then it is c-minimal.
\end{prop}
\begin{proof}
We write $(p,\theta)$ for $\psi(w)$
and assume that $w$ is minimal but not c-minimal for a contradiction.
The assumption that $w$ is not c-minimal is equivalent to assuming that
$\psi(w)$ is not in $R_K(w)$.
This means that $\psi(w)$ is $i$-reducible for some $i$.
Thus $(p,\theta) \sim_K (p^\prime,\theta)$ for some nanophrase
 $p^\prime$ for which the $i$th component is empty.
\par 
Let $O$ equal $\{i\}$.
Then by Lemma~\ref{lem:simplify}, $p \sim f_O(p)$
 and so $(p,\theta) \sim_K (f_O(p),\theta)$.
Note that because $(p,\theta)$ is in the image of $\psi$, the $i$th
 component of $p$ is not empty.
Thus $\rank(p)$ is greater than $\rank(f_O(p))$.
\par
Let $(q,\theta^\prime)$ be the result of applying an $i$-reduction to
 $(f_O(p),\theta)$.
Then $\rank(q)$ is equal to $\rank(f_O(p))$ and so less than
 $\rank(p)$. 
As $(q,\theta^\prime) \sim_P (p,\theta)$, by
 Lemma~\ref{lem:hinvariance}, 
$\Omega((q,\theta^\prime)) \sim \Omega((p,\theta))$.
Now $\Omega((p,\theta))$ is $w$, so $\Omega((q,\theta^\prime)) \sim w$.
However, $\rank(\Omega((q,\theta^\prime)))$ is equal to $\rank(q)$
 which is less than $\rank(p)$, so $\hr(w)$ is less than $\rank(w)$
 contradicting the fact that $w$ is minimal.
\end{proof}
\begin{rem}
Note that if a nanoword $w$ is c-minimal it does not necessarily mean
 that it is minimal.
For example, suppose $w$ is a non-trivial minimal nanoword.
Then $w$ has the form $XyXz$ for some letter $X$ and some words $y$ and
 $z$.
Then the nanoword $AAXyXz$ where $\xType{A}$ is equal to $\xType{X}$ is
 clearly c-minimal but not minimal. 
\end{rem}
\begin{thm}\label{thm:hr}
Let $(\alpha,\tau,S)$ be a composite homotopy data
 triple and let its prime factors be denoted by
 $(\alpha_i,\tau_i,S_i)$.
Let $w$ be a nanoword over $\alpha$.
Then
\begin{equation}\label{eqn:compositehr}
\hr(w) = \sum_i \hr(P_{R,i}(w)).
\end{equation}
\end{thm}
\begin{proof}
Let $u$ be a minimal nanoword homotopic to $w$.
Then $\rank(u)$ is equal to $\hr(w)$. 
By Proposition~\ref{prop:reducedisminimal} $\psi(u)$ is in $R_K(w)$.
Then $\hr(P_{R,i}(u))$ is equal to $\hr(P_{R,i}(w))$ and clearly
\begin{equation}\label{eqn:compositehrbelow}
\hr(u) \geq \sum_i \hr(P_{R,i}(w)).
\end{equation}
\par
On the other hand, for each $i$, let $q_i$ be a minimal nanophrase such
 that $q_i \sim_i P_{R,i}(w)$.
Then let $(q, \theta_R(w))$ be the element of $\mathcal{P}_R(\alpha)$
 such that $s_i((q, \theta_R(w)))$ is equal to $q_i$ for all $i$.
Then $(q, \theta_R(w))$ is in $R_K(w)$.
Then by Lemma~\ref{lem:hinvariance}, $\Omega((q, \theta_R(w))) \sim u$ and so
\begin{equation}\label{eqn:compositehrabove}
\hr(u) \leq \sum_i \hr(P_{R,i}(w)).
\end{equation}
Combining the inequalities \eqref{eqn:compositehrbelow} and
 \eqref{eqn:compositehrabove} gives \eqref{eqn:compositehr}.
\end{proof}
\begin{rem}
If $(\alpha,\tau,S)$ is a prime homotopy data triple then the results of
 this section hold, but are trivial.
The set $\mathcal{P}_R(\alpha)$ consists of the element
$(\zerocomp,\theta_\emptyset)$, where $\theta_\emptyset$ is the empty
 map, and elements $(p,\theta: 1 \mapsto 1)$ where $p$ is a
 $1$-component nanophrase.
Thus $\theta_R(w)$ can take one of two values depending on whether $w$
 is homotopic to $\trivial$ or not.
However, to calculate the value of $\theta_R(w)$ for a given nanoword
 $w$ we must show whether $w$ is homotopic to $\trivial$ or not and so
 the invariant is trivial.
Similarly, if $w$ is not homotopic to $\trivial$, $P_{R,1}(\psi(w))$ is
 given by $w$, and again the invariant is trivial.
In this case, Theorem~\ref{thm:hr} is reduced to the tautological
 statement that $\hr(w)$ is equal to itself.
\end{rem}
We finish this section by briefly considering some symmetries of
nanowords defined by Turaev in \cite{Turaev:Words}.
We start by recalling some definitions.
Given a nanoword $w$, the \emph{opposite nanoword} of $w$ is the
nanoword given by taking the letters of $w$ in the reverse order and
preserving the projection of the letters.
We write $w^-$ for the opposite nanoword of $w$.
The \emph{inverse nanoword} of $w$, written $\overline{w}$, has the same
Gauss word as $w$ but the projections of the letters are composed with
$\tau$.
In other words, for all $A$ appearing in $w$,
$\xType{A}_{\overline{w}}$, the projection of $A$ in 
$\overline{w}$, is equal to $\tau(\xType{A}_w)$, where $\xType{A}_w$ is
the projection of $A$ in $w$.
\par
We extend these definitions to nanophrases.
Given a nanophrase $p$, the \emph{opposite nanophrase} of $p$, written
$p^-$ is the nanophrase given by taking the components in reverse order
and, for each component, reversing the order of the letters in the
component.
The \emph{inverse nanophrase} of $p$, written $\overline{p}$, has the
same Gauss phrase as $p$ but the projections of the letters are composed
with $\tau$.
These operations are involutions, so the opposite nanophrase of $p^-$ is
$p$ and the inverse nanophrase of $\overline{p}$ is $p$.
The two operations commute.
Note that for a $1$-component nanophrase, which we can consider to be a
nanoword, the definitions given for nanowords and nanophrases coincide.
\begin{ex}
Let $\alpha$ be the set $\{a,b\}$ and let $\tau$ be the map swapping $a$
 with $b$.
Let $p$ be the nanophrase $ABC|AC|B$ where $\xType{A}$ and $\xType{C}$
 are $a$ and $\xType{B}$ is $b$.
Then $p^-$ is $B|CA|CBA$ where the letters have the same projections as
 $p$.
The nanophrase $\overline{p}$ is $ABC|AC|B$ where $\xType{A}$ and
 $\xType{C}$ are $b$ and $\xType{B}$ is $a$.
\end{ex}
In \cite{Turaev:Words}, Turaev defined the following terms.
Given a homotopy, a nanoword $w$ is \emph{homotopically symmetric} (with
respect to that homotopy) if $w$ is homotopic to $w^-$.
Given a homotopy, a nanoword $w$ is \emph{homotopically skew-symmetric}
(with respect to that homotopy) if $w$ is homotopic to
$(\overline{w})^-$.
\par
We fix a homotopy data triple $(\alpha,\tau,S)$.
Let $\kappa_\tau$ be the involution $\tau \times \tau \times \tau$ on
$\alpha \times \alpha \times \alpha$,
and let $\kappa_i$ be the involution which maps
$(a,b,c)$ to $(c,b,a)$ on $\alpha \times \alpha \times \alpha$.
\par
Suppose $S$ is invariant under $\kappa_\tau$.
Then Turaev observed in \cite{Turaev:Words} that if $u$ and $v$ are 
homotopic under the homotopy given by $(\alpha,\tau,S)$, then
$\overline{u}$ and $\overline{v}$ are also homotopic under the same
homotopy.
Then it is easy to see that, for any nanoword $w$, $c_R(\overline{w})$
is equal to $c_R(w)$, $\theta_R(\overline{w})$ is equal to $\theta_R(w)$
and, for all $i$, 
\begin{equation*}
P_{R,i}(\overline{w}) = \overline{P_{R,i}(w)}. 
\end{equation*}
\par
Now suppose that $S$ is invariant under $\kappa_i$.
Then Turaev observed in \cite{Turaev:Words} that if $u$ and $v$ are 
homotopic under the homotopy given by $(\alpha,\tau,S)$, then
$u^-$ and $v^-$ are also homotopic under the same
homotopy.
Again, it is easy to see that, for any nanoword $w$, $c_R(w^-)$
is equal to $c_R(w)$, $\theta_R(w^-)(i)$ is equal to
$\theta_R(w)(c_R(w) + 1 - i)$ for all $i$ 
and 
\begin{equation*}
P_{R,i}(w^-) = (P_{R,i}(w))^-. 
\end{equation*}
for all $i$.
\par
So if $S$ is invariant under $\kappa_i$, $w$ can only be homotopically
symmetric if
\begin{equation}\label{eqn:symm}
(\theta_R(w))(i) = \theta_R(w)(c_R(w) + 1 - i)
\end{equation}
for all $i$.
As $\theta_R(w)$ is locally variable, this implies that $c_R(w)$ must be
odd.
Thus if $c_R(w)$ is even or \eqref{eqn:symm} does not hold for all $i$,
$w$ cannot be homotopically symmetric.
Similarly, if $S$ is invariant under $\kappa_i$ and $\kappa_\tau$, $w$
cannot be homotopically skew-symmetric if $c_R(w)$ is even or
\eqref{eqn:symm} does not hold for all $i$.
\section{Decidability}\label{sec:decidable-homotopies}
We say that a homotopy is \emph{reduction decidable} if we have a finite
time algorithm which, for any nanophrase $p$ and any integer $i$,
determines whether $p$ is $i$-reducible or not.
We say that a homotopy is \emph{equality decidable} if we have a finite
time algorithm which, for any two nanophrases $p$ and $q$, determines
whether $p$ is equivalent to $q$ or not.
\begin{prop}\label{prop:nanowordreduction}
Let $(\alpha,\tau,S)$ be a composite homotopy data
 triple and let its prime factors be denoted by
 $(\alpha_i,\tau_i,S_i)$.
Suppose, for each $i$, the homotopy given by
 $(\alpha_i,\tau_i,S_i)$ is reduction decidable.
Then there is a finite time algorithm which, for any nanoword over
 $\alpha$, $w$, produces a nanoword $R(w)$ such that $\psi(R(w))$ is in
 $R_K(w)$.
\end{prop}
\begin{proof}
We give a finite time algorithm to produce $R(w)$ from an arbitrary
 nanoword $w$.
\par
We first note that $\psi(w)$ is reducible if and only if
 $\psi_i(w)$ is reducible for some $i$.
As, for each $i$, the homotopy given by
 $(\alpha_i,\tau_i,S_i)$ is reduction decidable,
 we can determine whether or not $\psi(w)$ is reducible in finite time.
\par
If $\psi(w)$ is not reducible then $\psi(w)$ is in $R_K(w)$ and we
 define $R(w)$ to be $w$.
\par
If $\psi(w)$ is reducible, then it is $j$-reducible for some $j$.
Let $O$ be the subset of $\widehat{\nc(p)}$ consisting of the single
 element $j$.
Writing $(p,\theta)$ for $\psi(w)$, let $p^\prime$ be $f_O(\psi(w))$.
Then by Lemma~\ref{lem:simplify} and
 Lemma~\ref{lem:restrictionembedding}, 
$(p,\theta) \sim_K (p^\prime,\theta)$.
Let $c$ be the equivalence class under $\sim_K$ which contains $\psi(w)$
 and let $c^\prime$ be the equivalence class under $\sim_K$ which  
 contains $\psi(\chi(p^\prime))$.
It is clear that $c^\prime$ is derivable from $c$ by a series of one or
 more reductions.
We define $R(w)$ to be $R(\chi(p^\prime))$.
This definition is recursive, but the recursion will terminate because
 $\nc(\chi(p^\prime))$ is less than $\nc(w)$.
\end{proof}
We note that $R(w)$ is not necessarily well-defined as it depends on the
order that the reductions are made.
For example, consider the nanoword $ACADDBBC$ where $\xType{A}$ and
$\xType{B}$ are in $\alpha_1$ and $\xType{C}$ and $\xType{D}$ are in
$\alpha_2$.
Then $\psi(w)$ is $(A|C|A|DD|BB|C, \theta)$ where $\theta$ maps odd
numbers to $1$ and even numbers to $2$.
The $4$th and $5$th components of $\psi(w)$ are clearly reducible by an
H1 move.
If we apply the above process to the $4$th component we get
$(A|C|ABB|C, \theta^\prime)$ where $\theta^\prime$ maps odd numbers to
$1$ and even numbers to $2$.
If we apply the above process to the $5$th component we get
$(A|C|A|DDC, \theta^\prime)$.
Using the linking matrix described in Section~\ref{subsec:linkingmatrix}
we can show that both these elements are reduced.
Thus $R(w)$ can be $(A|C|ABB|C, \theta^\prime)$ or 
$(A|C|A|DDC,\theta^\prime)$ 
depending on the order we make the reductions.
\begin{thm}\label{thm:decidability}
Let $(\alpha,\tau,S)$ be a composite homotopy data
 triple and let its prime factors be denoted by
 $(\alpha_i,\tau_i,S_i)$.
Suppose, for each $i$, the homotopy given by
 $(\alpha_i,\tau_i,S_i)$ is reduction decidable and
 equality decidable.
Then, there is a finite time algorithm which for any two nanowords over
 $\alpha$, $w$ and $w^\prime$, determines whether or not they are
 equivalent under the homotopy given by
 $(\alpha,\tau,S)$.
\end{thm}
\begin{proof}
Let $w$ and $w^\prime$ be two nanowords in $\mathcal{\alpha}$.
By Proposition~\ref{prop:nanowordreduction} we can
 calculate $R(w)$ and $R(w^\prime)$ in finite time.
We write $(p,\theta)$ for $R(w)$ and $(p^\prime,\theta^\prime)$ for
 $R(w^\prime)$.
Now, as $\psi(R(w))$ and $\psi(R(w^\prime))$ are irreducible, if
 $\theta$ and $\theta^\prime$ are not equal, $w$ and $w^\prime$ are not
 equivalent.
\par
If $\theta$ and $\theta^\prime$ are equal, then by
Corollary~\ref{cor:completeinvariant},
$\psi(w) \sim \psi(w^\prime)$ if and only if
$P_{R,i}(w) \sim_i P_{R,i}(w^\prime)$ for all $i$.
Now, for all $i$, $P_{R,i}(w) \sim_i P_{R,i}(w^\prime)$ can be
 determined in finite time because $\sim_i$ is equality decidable.
\end{proof}
\begin{rem}
Recall the definitions of $(\alpha_G,\tau_G,S_G)$ and
 $(\alpha_F,\tau_F,S_F)$ from Example~\ref{ex:diagonalshdt}.
By Theorem~\ref{thm:decidability} and
 Example~\ref{ex:diagonalfactorization} we may conclude the following.
If $(\alpha_G,\tau_G,S_G)$ and $(\alpha_F,\tau_F,S_F)$ are both
 reduction and equality decidable then for any $S$ diagonal homotopy
 there is a finite time algorithm which can determine homotopy
 equivalence of arbitrary pairs of nanowords.
Whether $(\alpha_G,\tau_G,S_G)$ or $(\alpha_F,\tau_F,S_F)$ are
 reduction or equality decidable is an open question.
\end{rem}
We say that a homotopy has a \emph{normal form} if there exists a finite
time algorithm taking a nanophrase $p$ and producing a nanophrase $n(p)$
which satisfies the following conditions:
\begin{enumerate}
\item
$n(p) \sim p$ for all $p$,
\item
if $p \sim p^\prime$, then $n(p)$ equals $n(p^\prime)$ for all $p$ and
     $p^\prime$.
\end{enumerate}
If, for all $p$, $\rank(n(p))$ is equal to $\hr(p)$, then we say that $n$ is
a \emph{minimizing normal form}.
\begin{prop}\label{prop:mnf-implies-decidable}
If a homotopy has a minimizing normal form $n$, it is reduction decidable and
 equality decidable.
\end{prop}
\begin{proof}
The fact that the homotopy is equality decidable immediately follows
 from the existance of a normal form (minimizing or not).
Now $p$ is $i$-reducible if and only if $n(p)$ is $i$-reducible.
If the $i$th component of $n(p)$ is $\trivial$, then $n(p)$ is
 $i$-reducible by definition.
On the other hand, suppose that the $i$th component of $n(p)$ is not
 $\trivial$.
Then as $n$ is a minimizing normal form, $\rank(n(p))$ is equal to
 $\hr(p)$ and so by Lemma~\ref{lem:hr-emptycomponent}, $n(p)$ is not
 $i$-reducible.
Thus $p$ is $i$-reducible if and only if the $i$th component of $n(p)$
 is trivial, and so the homotopy is reduction decidable.
\end{proof}
Let $S_\trivial$ be the empty set.
Then any homotopy data triple $(\alpha,\tau,S_\trivial)$ gives a
homotopy where the third homotopy move is disallowed.
Thus the homotopy relation is generated by isomorphism and the first two
homotopy moves.
\par
The homotopy relation induces an equivalence relation on the set of
isomorphism classes of $\mathcal{P}(\alpha)$.
We call a first homotopy move or second homotopy move \emph{reducing} if
it removes letters from a nanophrase.
For two nanophrases $p$ and $q$ we write $p \geq q$ if either
$p$ is equal to $q$ or $q$ is derivable from $p$ by a sequence of
reducing homotopy moves.
The following lemma shows that for the homotopy given by
$(\alpha,\tau,S_\trivial)$,
the set of isomorphism classes satisfies the Confluence
condition of Newman's Diamond Lemma (Lemma~\ref{lem:diamond}).
\begin{lem}\label{lem:gauss-confluence}
Let $p$ be a nanophrase in $\mathcal{P}(\alpha)$.
Suppose $p_1$ and $p_2$ are the results of applying different reducing
 moves to $p$ under the homotopy given by
 $(\alpha,\tau,S_\trivial)$.
Then there exists a nanophrase $p^\prime$ such that $p_1 \geq p^\prime$
 and $p_2 \geq p^\prime$.
\end{lem}
\begin{proof}
If the sets of letters involved in the two moves do not intersect, then
 it is clear that the moves can be applied in either order and give the
 same result in both cases.
So we only need to check the cases where the sets of letters do
 intersect.
\par
If both moves are first homotopy moves, the set of letters can not
 intersect. 
\par
If one move is a first homotopy move and the other a second homotopy
 move, then we must have the case $xABBAy$ where 
 $\xType{B} = \tau(\xType{A})$. 
Applying the second homotopy move gives $xy$.
On the other hand, applying the first homotopy move gives $xAAy$.
We can then apply another first homotopy move to get $xy$ which was the
 result of applying the second homotopy move.
Thus in this case the Confluence condition holds.
\par
If both moves are second homotopy moves, then we must have the case
 $xABCyCBAz$, where $\xType{A} = \tau(\xType{B}) = \xType{C}$.
Then one of the moves removes $A$ and $B$ to give $xCyCz$.
The other move removes $B$ and $C$ to give $xAyAz$.
However, these nanophrases are isomorphic because $\xType{A}$ and
 $\xType{C}$ are equal.
Thus in this case the Confluence condition also holds.
\end{proof}
As the rank of a nanophrase is finite and decreases under a reducing
move, the Finiteness condition of Newman's Diamond Lemma is also
satisfied. 
Thus, by Newman's Diamond Lemma (Lemma~\ref{lem:diamond}),
the homotopy given by $(\alpha,\tau,S_\trivial)$ has a normal form (up
to isomorphism).
Given a nanophrase $p$ in $\mathcal{P}(\alpha)$, the normal form is
calculated by applying reducing homotopy moves wherever they appear.
When no more reducing homotopy moves are applicable, the normal form has
been reached.
In particular, it is clear that the normal form is minimizing.
\par
We remark that Manturov noted the existance of this normal form in the
specific case where $\alpha$ contains a single element in
\cite{Manturov:freeknots}.
\par
By Proposition~\ref{prop:mnf-implies-decidable}, we have the
following proposition.
\begin{prop}\label{prop:emptys-decidable}
Let $(\alpha,\tau,S_\trivial)$ be a homotopy data triple.
The homotopy that it gives 
is reduction decidable and equality decidable.
\end{prop}
In particular, there are exactly two prime data triples for which $S$ is
$S_\trivial$.
The first is $(\{a\},\tau_{id},S_\trivial)$ where $\tau_{id}$ is the
identity map.
The second is $(\{a,b\},\tau_a,S_\trivial)$ where $\tau_{a}$ maps $a$ to
$b$ (and therefore maps $b$ to $a$).
\begin{rem}
We note that Theorem~3.8 of \cite{Kadokami:Non-triviality} would imply
 the existence of a minimizing normal form for the homotopy given by
$(\alpha_F,\tau_F,S_F)$ (defined in Example~\ref{ex:diagonalshdt}).
However, we found a counter-example to the theorem which we explained in
 \cite{Gibson:tabulating-vs}.
\end{rem}
\section{Detecting irreducibility}\label{sec:detecting}
In this section we look at ways of determining irreducibility of
components of nanophrases.
\subsection{Component length}
In \cite{Gibson:gauss-phrase}, we noted that the number of letters in a
component modulo $2$ is a homotopy invariant of Gauss phrases.
Fukunaga made the same observation for nanophrases under homotopies with
diagonal $S$ in \cite{Fukunaga:nanophrases}.
In fact, the number of letters in a component modulo $2$ is a homotopy
invariants of nanophrases for any choice of $S$.
\par
Therefore, if the $i$th component of a nanophrase $p$ contains an odd
number of letters it cannot be $i$-reducible.
\begin{ex}
Let $p$ be the nanophrase $ABC|A|B|C$.
Each component of $p$ has an odd number of letters.
Then, irrespective of the projections of the letters $A$, $B$ and $C$,
 $p$ is irreducible under any nanophrase homotopy.
\end{ex}
\subsection{Linking matrix}\label{subsec:linkingmatrix}
In \cite{Gibson:gauss-phrase}, we defined a homotopy invariant of Gauss
phrases called the \emph{linking matrix}.
In \cite{Fukunaga:nanophrases2}, Fukunaga defined the \emph{linking
vector} of a nanophrase and proved that it is invariant under homotopies
with diagonal $S$.
As we observed in \cite{Gibson:gauss-phrase}, our linking matrix is
equivalent to the linking vector of a Gauss phrase.
\par
Using our linking matrix terminology, we now recall Fukunaga's linking
vector invariant. We observe that it is invariant for any choice of $S$
in the homotopy triple, not just for diagonal $S$.
\par
We fix a homotopy data triple $(\alpha,\tau,S)$.
Let $\pi$ be the multiplicative abelian group generated by elements of
$\alpha$ such that $a \tau(a) = 1$ for all $a$ in $\alpha$.
This group was defined by Turaev in \cite{Turaev:Words}.
\par
The \emph{linking matrix} $L(p)$ of an $n$-component
nanophrase is a symmetric $n \times n$ matrix with elements in $\pi$.
The elements of $L(p)$ are denoted by  $l_p(i,j)$.
We define $l_p(i,i)$ to be $1$ for all $i$ and define $l_p(i,j)$ by
\begin{equation*}
l_p(i,j) = \prod_{X \in \mathcal{A}_{ij}} \xType{X}
\end{equation*} 
where $\mathcal{A}_{ij}$ denotes the subset of letters appearing in
$p$ which appear both in the $i$th and $j$th components of $p$.
Note that by definition, $l_p(i,j)$ is equal to $l_p(j,i)$ for all $i$
and $j$.
Fukunaga's linking vector is the vector 
$( l_p(1,2),l_p(1,3),\dotsc,l_p(1,n),l_p(2,3),\dotsc,l_p(n-1,n) )$.
In the case where $S$ is diagonal, Fukunaga showed in
\cite{Fukunaga:nanophrases2} that the linking matrix $L(p)$ is
invariant.
This is easily extended to the general case.
\begin{prop}
The linking matrix $L(p)$ is a homotopy invariant of $p$.
\end{prop}
\begin{proof}
After noting that the sets $\mathcal{A}_{ij}$ do not change under the third
 homotopy move for any $S$, the proof is the same as in 
\cite{Fukunaga:nanophrases2}.
\end{proof}
\begin{ex}\label{ex:linking1}
Let $p$ be the nanophrase $ABC|AC|B$ where $\xType{A}$ is $a$,
 $\xType{B}$ is $b$ and $\xType{C}$ is $c$.
Then $L(p)$ is given by
\begin{equation*}
\begin{pmatrix}
1 & ac & b \\
ac & 1 & 1 \\
b & 1 & 1
\end{pmatrix}.
\end{equation*} 
\end{ex}
Let $L_i(p)$ be the $i$th row of $L(p)$.
We call $L_i(p)$ the \emph{linking vector} of the $i$th component of
$p$.
Note that if the $i$th component of $p$ is empty, then every element of
$L_i(p)$ will be $1$.
Thus if there is an element in $L_i(p)$ which is not $1$, then we can
conclude that $p$ is not $i$-reducible.
\begin{ex}
Consider again the nanophrase $p$ given in Example~\ref{ex:linking1}.
From the linking matrix of $p$ we can see that both the first and third
 rows contain the element $b$ which is not $1$ in $\pi$.
Thus $p$ is not $1$-reducible or $3$-reducible.
Furthermore, if $a$ is not equal to $\tau(c)$ then $ac$ is not $1$ in
 $\pi$ and $p$ is not $2$-reducible either.
\end{ex}
\subsection{Subphrase invariants}\label{subsec:subphrase}
Let $p$ be an $n$-component nanophrase and let $i$ be an element of
$\hat{n}$. 
Let $O$ be some subset of $\hat{n} - \{i\}$.
Then, by Lemma~\ref{lem:subphraseinvariance}, $x(p,O)$ is an invariant
of $p$.
As $i$ is not in $O$, $x(p,O)$ contains a component corresponding to the
$i$th component of $p$.
Let $j$ be the index of this component in $x(p,O)$.
Then if $x(p,O)$ is not $j$-reducible, $p$ is not $i$-reducible.
\par
In particular, when $O$ is the set $\hat{n}-\{i\}$, then $x(p,O)$ is a
nanoword which we denote $w_i(p)$.
If $w_i(p)$ is not contractible, $p$ cannot be $i$-reducible.
\begin{ex}
Let $p$ be a nanophrase over $\alpha_1$ given by $ABACDECBDF|EF$.
Then $w_1(p)$ is $ABACDCBD$ and $w_2(p)$ is $\trivial$.
\par
We consider $p$ under the homotopy given by the homotopy data triple 
$(\alpha_G,\tau_G,S_G)$ given in Example~\ref{ex:diagonalshdt}.
This homotopy is the open Gauss word homotopy.
By Example~6.2 of \cite{Gibson:gauss-word} we know that $ABACDCBD$ is
 not contractible.
Therefore $p$ is not $1$-reducible.
\end{ex}
\subsection{$\SoAll$ invariant}
We defined the $S_o$ invariant for Gauss phrases in
\cite{Gibson:gauss-phrase}.
Here we generalize the invariant to nanophrases for arbitrary $\alpha$,
$\tau$ and $S$.
We rename the invariant $\SoAll$ in order to avoid confusion with
Fukunaga's generalization of the $S_o$ invariant
\cite{Fukunaga:gen-app}.
\par
We write $K_n$ for $\cyclicproduct{2}{n}$.
Let $p$ be an $n$-component nanophrase and let $i$ be an element of
$\hat{n}$.
Let $\mathcal{A}_{ii}(p)$ be the subset of letters appearing in $p$ such
that both occurrences of the letter appear in the $i$th component of
$p$.
Let $A$ be a letter in $\mathcal{A}_{ii}(p)$.
Then the $i$th component of $p$ has the form $xAyAz$ for some, possibly
empty, words $x$, $y$ and $z$.
We define the \emph{linking vector} of $A$, written $\gvector{l(A)}$, to
be a vector $\cvector{v}$ in $K_n$.
The $j$th element of $\cvector{v}$ is, modulo $2$, the number of
letters that appear once in $y$ and for which the other occurrence of the
letter appears in the $j$th component of $p$.
\par
Let $\mathcal{A}_{i,a}$ be the subset of $\mathcal{A}_{ii}(p)$
consisting of letters $A$ such that $\xType{A}$ is equal to $a$.
For each element $a$ in $\alpha$ we define a map $d_{i,a}$ from 
$K_n - \{\cvector{0}\}$ to $\Z$, given by
\begin{equation*}
d_{i,a}(\cvector{v}) = \sharp \lbrace X\in \mathcal{A}_{i,a} \; | \;
 \gvector{l(X)} = \cvector{v} \rbrace,
\end{equation*}
where $\sharp$ means the number of elements in the set.
\begin{prop}
If $\tau(a)$ is equal to $a$ then $d_{i,a}(\cvector{v})$ modulo $2$ is a
 homotopy invariant of $p$.
If $\tau(a)$ is not equal to $a$ then 
$d_{i,a}(\cvector{v}) - d_{i,\tau(a)}(\cvector{v})$
 is a homotopy invariant of $p$.
\end{prop}
\begin{proof}
The proof is similar to that of the invariance of $S_o$ for Gauss
 phrases in \cite{Gibson:gauss-phrase}.
\par
The point is that under a homotopy move, the linking vector of any
 letter in $\mathcal{A}_{ii}(p)$ does not change and so is counted the same
 in $d_{i,a}$, unless the letter is added or removed by the move.
\par
If the letter is added or removed by the first homotopy move then the
 linking vector of the letter is $\cvector{0}$.
As $\cvector{0}$ is not in the domain of $d_{i,a}$, $d_{i,a}$ is
 unchanged.
\par
Under the second homotopy move two letters are added or removed.
If one of the letters is in $\mathcal{A}_{ii}(p)$ then both are.
In this case, both letters have the same linking vector $\cvector{v}$.
If the projection of neither letter is $a$ then $d_{i,a}$ is
 unchanged.
If not, the projection of one letter is $a$ and the other is $\tau(a)$. 
If $a$ is equal to $\tau(a)$, then we count the vector twice in
 $d_{i,a}$ and so, modulo $2$, $d_{i,a}$ is unchanged.
If $a$ is not equal to $\tau(a)$, then one letter contributes to
 $d_{i,a}$ and the other letter contributes to $d_{i,\tau(a)}$.
Thus $d_{i,a}(\cvector{v}) - d_{i,\tau(a)}(\cvector{v})$ is unchanged by
 the move.
\end{proof}
We define $\SoAll^{i,a}(\cvector{v})$ to be $d_{i,a}(\cvector{v})$ modulo $2$
if $a$ is equal to $\tau(a)$ and 
$d_{i,a}(\cvector{v}) - d_{i,\tau(a)}(\cvector{v})$ otherwise.
We fix a subset $\alpha_0$ of $\alpha$ which contains exactly one
element from each orbit of $\tau$.
Turaev calls this an \emph{orientation} of $\alpha$.
Now note that if $\tau(a)$ is not equal to $a$ then
$\SoAll^{i,\tau(a)}(\cvector{v})$ is equal to
$-\SoAll^{i,a}(\cvector{v})$ for all $\cvector{v}$ in 
$K_n - \{\cvector{0}\}$.
Thus $\SoAll$ is completely defined by the maps $\SoAll^{i,a}$ for $a$ in
$\alpha_0$.
\begin{ex}\label{ex:so}
Let $\alpha$ be the set $\{a,b,c\}$.
Let $\tau$ map $a$ to $b$ and $c$ to itself.
Let $S$ be the set $\{(a,b,c),(c,b,a)\}$.
\par
Let $p$ be a nanophrase over $\alpha$ given by $ADBAEBFG|CDCF|EG$ where
 $\xType{A}=\xType{D}=a$, $\xType{B}=\xType{F}=b$ and
 $\xType{C}=\xType{E}=\xType{G}=c$.
Note that the linking matrix of $p$ is trivial.
Thus, using the linking matrix, $p$ is indistinguishable from the
 trivial $3$-component nanophrase $\trivial|\trivial|\trivial$.
\par
Now $\mathcal{A}_{11}(p)$ is the set $\{A,B\}$, $\mathcal{A}_{22}(p)$
 is the set $\{C\}$ and $\mathcal{A}_{33}(p)$ is empty.
Calculating the linking vectors of $A$, $B$ and $C$, we see that
$\gvector{l(A)}$ is $(1,1,0)$, $\gvector{l(B)}$, is $(1,0,1)$. 
and $\gvector{l(C)}$, is $(1,0,0)$.
Then, $d_{1,a}(\cvector{v})$ is equal to $1$ if $\cvector{v}$ is equal
 to $(1,1,0)$ and $0$ otherwise.
Similarly, $d_{1,b}(\cvector{v})$ is equal to $1$ if $\cvector{v}$ is
 equal to $(1,0,1)$ and $0$ otherwise. 
Furthermore, $d_{2,c}(\cvector{v})$ is equal to $1$ if $\cvector{v}$ is
 equal to $(1,0,0)$ and $0$ otherwise. 
All the other maps $d_{i,x}$ map every vector to $0$.
\par
Therefore $\SoAll^{1,a}(\cvector{v})$ is equal to $1$ if $\cvector{v}$ is
 equal to $(1,1,0)$, $-1$ if $\cvector{v}$ is equal to $(1,0,1)$ and $0$
 otherwise.
As $b$ is equal to $\tau(a)$, $\SoAll^{1,b}(\cvector{v})$ is equal to
 $-\SoAll^{1,a}(\cvector{v})$.
Thus $\SoAll^{1,b}(\cvector{v})$ is equal to $-1$ if $\cvector{v}$ is equal
 to $(1,1,0)$, $1$ if $\cvector{v}$ is equal to $(1,0,1)$ and $0$
 otherwise.
Also, $\SoAll^{2,c}(\cvector{v})$ is equal to $1$ if $\cvector{v}$ is equal
 to $(1,0,0)$ and $0$ otherwise.
For all other pairs $(i,x)$, $\SoAll^{i,x}(\cvector{v})$ is $0$ for all
 $\cvector{v}$.
\par
Since $\trivial|\trivial|\trivial$ has trivial maps for all $\SoAll^{i,x}$,
 we may conclude that $p$ and $\trivial|\trivial|\trivial$ are not
 homotopic.
\end{ex}
\begin{rem}
In the case of Gauss phrases, $\alpha$ contains a single element $a$ and
 $\tau$ is the identity map.
Thus for the $i$th component we just have a single map $\SoAll^{i,a}$ from
 $K_n - \{\cvector{0}\}$ to $\cyclic{2}$.
In \cite{Gibson:gauss-phrase} we defined $B_i(p)$ to be the preimage of
 $1$ under the map $\SoAll^{i,a}$ and then defined $S_o$ to be an $n$-tuple
 where the $i$th element is $B_i(p)$.
\end{rem}
\begin{rem}
In \cite{Fukunaga:gen-app} Fukunaga generalized the Gauss phrase $S_o$
 invariant to any homotopy with diagonal $S$.
His invariant is also called $S_o$.
However, his invariant is equivalent to the invariant we call $\SoDiag$
 which is defined in Section~\ref{subsec:suinvariant}.
\end{rem} 
If, for a given $p$, $\mathcal{A}_{ii}(p)$ is empty then
$\SoAll^{i,a}(\cvector{v})$ is $0$ for all vectors $\cvector{v}$.
Thus $p$ is not $i$-reducible if there exists some $a$ in $\alpha$ and
some vector $\cvector{v}$ such that $\SoAll^{i,a}(\cvector{v})$ is
non-zero.
\begin{ex}
Consider again the nanophrase $p$ defined in Example~\ref{ex:so}.
From our calculations in that example we know that $\SoAll^{1,a}((1,1,0))$
 is non-zero.
We also know that $\SoAll^{2,c}((1,0,0))$ is non-zero.
Therefore, $p$ is neither $1$-reducible nor $2$-reducible.
\par
Note that since the linking matrix of $p$ is trivial, we could not
 determine this information from the linking matrix.
\end{ex}
\subsection{$\SoDiag$ invariant}\label{subsec:suinvariant}
In \cite{Turaev:Words} Turaev defined the self-linking function of a
nanoword, an invariant for nanowords under homotopies with diagonal
$S$.
This invariant is a generalization of the $u$-polynomial defined by
Turaev for flat virtual knots (also known as virtual strings)
\cite{Turaev:2004}.
Based on the constructions of the self-linking function for nanowords
and the invariant $\SoAll$ we define a new invariant for nanophrases under
homotopies with diagonal $S$ which we call $\SoDiag$.
\par
Let $\pi$ be the group defined in Section~\ref{subsec:linkingmatrix}.
Let $p$ be an $n$-component nanophrase and $i$ be an integer in
$\hat{n}$.
Let $A$ and $B$ be two letters in $\mathcal{A}_{ii}(p)$.
We define $n(A,B)$ to be $1$ if the $i$th component of $p$ has the form
$\dotso A \dotso B \dotso A \dotso B \dotso$,
$n(A,B)$ to be $-1$ if the $i$th component of $p$ has the form
$\dotso B \dotso A \dotso B \dotso A \dotso$
and $n(A,B)$ to be $0$ otherwise.
In particular $n(A,A)$ is defined to be $0$ for all $A$ in
$\mathcal{A}_{ii}(p)$.
\par
For $A$ in $\mathcal{A}_{ii}(p)$, we define $\gvector{l_u(A)}$ to be a vector in
$\pi^n$.
The $i$th element of $\gvector{l_u(A)}$ is defined to be
\begin{equation*}
\prod_{B \in \mathcal{A}_{ii}(p)} \xType{B}^{n(A,B)}.
\end{equation*}
Writing the $i$th component of $p$ as $xAyAz$ for some, possibly empty,
words $x$, $y$ and $z$, we define the $j$th element of
$\gvector{l_u(A)}$ ($j$ not equal to $i$) to be
\begin{equation*}
\prod_{B \in \mathcal{A}_{ij}(p), B \in y} \xType{B}.
\end{equation*}
\par
Then for each $a$ in $\alpha$ and for each $i$, let $\mathcal{A}_{i,a}$
be the subset of $\mathcal{A}_{ii}(p)$ consisting of letters $A$ such
that $\xType{A}$ is equal to $a$.
Let $\cvector{1}$ be the trivial vector in $\pi^n$ for which every
element is $1$.
For each element $a$ in $\alpha$ we define a map $e_{i,a}$ from 
$\pi^n - \{\cvector{1}\}$ to $\Z$, given by
\begin{equation*}
e_{i,a}(\cvector{v}) = \sharp \lbrace X\in \mathcal{A}_{i,a} \; | \;
 \gvector{l_u(X)} = \cvector{v} \rbrace,
\end{equation*}
where $\sharp$ means the number of elements in the set.
\begin{rem}\label{rem:selflinkingclass}
When $n$ is $1$, $p$ is a nanoword.
Then the self-linking class $[a]_p$ defined in \cite{Turaev:Words} can
 be derived from $e_{1,a}$ by
\begin{equation*}
[a]_p = \sum_{\cvector{v} \in \pi - \{\cvector{1}\}}
 e_{1,a}(\cvector{v})\cvector{v}. 
\end{equation*} 
\end{rem}
In light of Remark~\ref{rem:selflinkingclass}, the following proposition
is a generalization of Theorem~6.1.1 in \cite{Turaev:Words}.
\begin{prop}
If $\tau(a)$ is equal to $a$ then $e_{i,a}$ modulo $2$ is a homotopy
 invariant of $p$.
If $\tau(a)$ is not equal to $a$ then $e_{i,a} - e_{i,\tau(a)}$ is
 a homotopy invariant of $p$.
\end{prop}
\begin{proof}
It is sufficient to prove invariance under a single isomorphism or
 homotopy move.
\par
It is clear that $\gvector{l_u(A)}$ does not change under isomorphism.
Thus $e_{i,a}$ is unchanged under isomorphism.
\par
We now check the first homotopy move.
Suppose $A$ is the letter removed under the move.
Then for any other letter $B$, $\gvector{l_u(B)}$ is unchanged by the
 removal of $A$.
On the other hand, $\gvector{l_u(A)}$ is the trivial vector
 $\cvector{1}$ and so $\gvector{l_u(A)}$ cannot be in the domain of
 $e_{i,a}$.
Thus $e_{i,a}$ does not change under the first homotopy move.
\par
We now check the second homotopy move.
Suppose $A$ and $B$ are the letters removed under the move.
Let $C$ be some other letter and consider how $\gvector{l_u(C)}$ is
 affected by the move.
We note that if $A$ contributes to $\gvector{l_u(C)}$ then so does $B$.
Since $\xType{A} \cdot \xType{B}$ equals $1$, if $A$ and $B$ contribute
 to $\gvector{l_u(C)}$, their contributions cancel.
Thus $\gvector{l_u(C)}$ is unchanged by the removal of $A$ and $B$.
On the other hand, $A$ is in $\mathcal{A}_{ii}(p)$ if and only if $B$
 is.
In the case where they are both in $\mathcal{A}_{ii}(p)$, $\gvector{l_u(A)}$ is
 equal to $\gvector{l_u(B)}$.
If neither $A$ nor $B$ project to $a$, $e_{i,a}$ is unchanged under
 the move.
Now if the projection of one of $A$ or $B$ is $a$, then the projection
 of the other is $\tau(a)$.
Without loss of generality, we may assume that $\xType{A}$ is $a$.
If $a$ is equal to $\tau(a)$ then both $A$ and $B$ contribute to
 $e_{i,a}(\gvector{l_u(A)})$ and, modulo $2$, the contributions cancel.
If $a$ is not equal to $\tau(a)$ then $A$ contributes to
 $e_{i,a}(\gvector{l_u(A)})$ and $B$ contributes to
 $e_{i,\tau(a)}(\gvector{l_u(A)})$.
Thus $e_{i,a}(\gvector{l_u(A)}) - e_{i,\tau(a)}(\gvector{l_u(A)})$ is
 unchanged by the move.
\par
We now check the third homotopy move.
Suppose $A$, $B$ and $C$ are the letters involved in the move and 
before the move we have the pattern $xAByACzBCt$.
For any other letter $D$, if the letters $A$, $B$ or $C$ do not
 contribute to $\gvector{l_u(D)}$, then $\gvector{l_u(D)}$ is unaffected
 by the move.
If the letters $A$, $B$ or $C$ do contribute to $\gvector{l_u(D)}$,
 then, as $\pi$ is abelian, the order of the contributing letters is
 irrelevant and so $\gvector{l_u(D)}$ is unchanged by the move.
\par
Now consider $A$.
Before the move $B$ contributes to $\gvector{l_u(A)}$ and $C$ does not.
After the move $C$ contributes to $\gvector{l_u(A)}$ and $B$ does not.
However, $\xType{B}$ is equal to $\xType{C}$ and so the contributions of
 $B$ and $C$ to $\gvector{l_u(A)}$ are equal.
Thus $\gvector{l_u(A)}$ does not change under the move.
\par
Now consider $B$.
Before the move, both $A$ and $C$ contribute to $\gvector{l_u(B)}$ but
 their contributions cancel.
After the move, neither $A$ nor $C$ contribute to $\gvector{l_u(B)}$.
Thus $\gvector{l_u(B)}$ does not change under the move.
\par
The case of $C$ is symmetric to the case of $A$ and so
 $\gvector{l_u(C)}$ does not change under the move.
\par
Thus $e_{i,a}$ is invariant under the third homotopy move.
\end{proof}
We define $\SoDiag^{i,a}(\cvector{v})$ to be $e_{i,a}(\cvector{v})$ modulo $2$
if $a$ is equal to $\tau(a)$ and 
$e_{i,a}(\cvector{v}) - e_{i,\tau(a)}(\cvector{v})$ otherwise.
As we did for the $\SoAll$ invariant, we pick an orientation $\alpha_0$ of
$\alpha$ ($\alpha_0$ contains exactly one element from each orbit of
$\alpha$ under $\tau$).
Now note that if $\tau(a)$ is not equal to $a$ then
$\SoDiag^{i,\tau(a)}(\cvector{v})$ is equal to
$-\SoDiag^{i,a}(\cvector{v})$ for all $\cvector{v}$ in 
$\pi^n - \{\cvector{1}\}$.
Thus $\SoDiag$ is completely determined by the maps $\SoDiag^{i,a}$ for $a$ in
$\alpha_0$.
\begin{ex}
Let $\alpha$ be the set $\{a,b,c,d\}$.
Let $\tau$ map $a$ to $b$ and $c$ to $d$.
Let $S$ be the diagonal of $\alpha$.
\par
Let $p$ be a nanophrase over $\alpha$ given by $ACDEABFB|CE|DF$ where
$\xType{A}=\xType{C}=\xType{E}=a$, $\xType{B}=c$,
 $\xType{D}=\xType{F}=d$.
Let $q$ be the nanophrase $CDEF|CE|DF$ where the projections are the
 same as for $p$.
Then it is easy to check that the linking matrix cannot distinguish $p$
 and $q$.
\par
We calculate $\SoDiag$ for $p$.
We see that $\gvector{l_u(A)}$ is $(1,a^2,c^{-1})$ and
 $\gvector{l_u(B)}$ is $(1,1,c^{-1})$.
Thus $\SoDiag^{1,a}(\cvector{v})$ is $1$ if $\cvector{v}$ is
 $(1,a^2,c^{-1})$ and $0$ otherwise.
Also, $\SoDiag^{1,c}(\cvector{v})$ is $-1$ if $\cvector{v}$ is
 $(1,1,c^{-1})$ and $0$ otherwise.
For all other pairs $(i,x)$, $\SoDiag^{i,x}$ is $0$ for all vectors.
\par
On the other hand, for $q$, the maps $\SoDiag^{i,x}$ are trivial for all $i$
 and $x$.
Thus $p$ and $q$ are not homotopic.
\end{ex}
\begin{rem}
Having written this section we discovered that Fukunaga had also
 independently generalized the Gauss phrase $S_o$ invariant to any
 homotopy with diagonal $S$ in \cite{Fukunaga:gen-app}.
His invariant is called $S_o$.
\par
The initial definition of Fukunaga's $S_o$ invariant, appearing in an
 early version of \cite{Fukunaga:gen-app}, was weaker than our
 $\SoDiag$ invariant.
However, we realised that by a slight modification of the definition,
 Fukunaga's $S_o$ could be strengthened.
This modified definition is now the one that appears in
 \cite{Fukunaga:gen-app}.
In Section~\ref{subsec:fukunagaequiv} we will show that our $\SoDiag$
 invariant is equivalent to Fukunaga's $S_o$ invariant.
\end{rem} 
For any nanophrase we can calculate $\SoAll^{i,a}$ from $\SoDiag^{i,a}$.
To do so, we define a map $x$ from $\pi$ to $\cyclic{2}$ as follows.
Any element $g$ of $\pi$ can be written uniquely in the form
\begin{equation}\label{eqn:pi-element}
\prod_{a \in \alpha_0} a^{i_a}
\end{equation}
where $i_a$ is an integer depending on $a$.
Then we define $x(g)$ to be the sum of the exponents in
Equation~\eqref{eqn:pi-element} modulo $2$, that is
\begin{equation*}
x(g) = \sum_{a \in \alpha_0} i_a \mod 2.
\end{equation*}
The map $x$ then induces a map from $\pi^n$ to $K^n$, which we also
call $x$, by applying $x$ elementwise.
For a vector $\cvector{v}$ in $K^n$, let $I_v$ be the preimage of
$\cvector{v}$ under $x$.
Then 
\begin{equation*}
d_{i,a}(\cvector{v}) = \sum_{\cvector{u} \in I_v} e_{i,a}(\cvector{u}).
\end{equation*}
Thus if $a$ is equal to $\tau(a)$, 
\begin{equation*}
V^{i,a}(\cvector{v}) = \sum_{\cvector{u} \in I_v} U^{i,a}(\cvector{u})
 \mod 2,
\end{equation*}
and if $a$ is not equal to $\tau(a)$,
\begin{equation*}
V^{i,a}(\cvector{v}) = \sum_{\cvector{u} \in I_v} U^{i,a}(\cvector{u}).
\end{equation*}
\par
We have already noted in Example~\ref{ex:diagonalshdt} that there are only
two prime homotopy data triples for which $S$ is diagonal.
For $(\alpha_G,\tau_G,S_G)$ (using the notation of
Example~\ref{ex:diagonalshdt}), the $\SoAll$ invariant and $\SoDiag$ invariant
are equivalent.
For $(\alpha_F,\tau_F,S_F)$, the $\SoDiag$ invariant is stronger than the
$\SoAll$ invariant.
\begin{ex}
Let $p$ be the nanophrase $ABCA|BC$ where $\xType{A}$, $\xType{B}$ and
 $\xType{C}$ are all $a$ an element in $\alpha_F$.
\par
We calculate the $\SoAll$ invariant for $p$.
As the linking vector for $A$, $\gvector{l(A)}$ is $(0,0)$, and $A$ is
 the only letter in $p$ for which both occurences appear in the same
 component, we conclude that $d_{i,a}(\cvector{v})$ is $0$
for $i$ equal to $1$ or $2$ and for all $\cvector{v}$ in
$K_n - \{\cvector{0}\}$.
So $\SoAll^{i,a}(\cvector{v})$ is $0$ 
for $i$ equal to $1$ or $2$ and for all $\cvector{v}$ in
$K_n - \{\cvector{0}\}$.
Thus, using the $\SoAll$ invariant, $p$ is indistinguishable from the
 nanophrase $\trivial|\trivial$.
\par
We now calculate the $\SoDiag$ invariant for $p$.
In this case, the linking vector for $A$, $\gvector{l_u(A)}$ is
 $(1,a^2)$.
Note that as $a$ is not equal to $\tau(a)$, $a^2$ is not equal to $1$,
 so $(1,a^2)$ is a non-trivial vector in $\pi^2$.
Thus $e_{1,a}((1,a^2))$ is equal to $1$ and so $\SoDiag^{1,a}((1,a^2))$ is
 equal to $1$.
Therefore $p$ is not homotopic to $\trivial|\trivial$.
\par
For completeness we note that
for all other vectors $\cvector{v}$ in $\pi^n - \{\cvector{1}\}$,
$\SoDiag^{1,a}(\cvector{v})$ is equal to $0$.
For all vectors $\cvector{v}$ in $\pi^n - \{\cvector{1}\}$,
$\SoDiag^{2,a}(\cvector{v})$ is equal to $0$.
\end{ex}
For a given $p$, if $\mathcal{A}_{ii}(p)$ is empty then for all
$\cvector{v}$, $\SoDiag^{i,a}(\cvector{v})$ is $0$.
Thus $p$ cannot be $i$-reducible if there exists some $a$ in $\alpha$
and some vector $\cvector{v}$ for which $\SoDiag^{i,a}(\cvector{v})$ is
non-zero.
Indeed, the $\SoDiag$ invariant can sometimes detect $i$-irreducibility
where the $\SoAll$ invariant cannot.
\par
Let $p$ be a nanophrase and let $w$ be $w_i(p)$ for some $i$, where
$w_i(p)$ is the nanoword derived from the $i$th component of $p$ which
was defined in Section~\ref{subsec:subphrase}.
For any $g$ in $\pi$ let $I_g$ be the set
\begin{equation*}
\{\cvector{v} \in \pi^n - \{\cvector{1}\} \;|\;
 \left[\cvector{v}\right]_i = g \}
\end{equation*}
where $\left[\cvector{v}\right]_i$ denotes the $i$th element of
$\cvector{v}$.
Then it is easy to check that $\SoDiag^{1,a}(g)$ for $w$ is given by
\begin{equation*}
\sum_{\cvector{v} \in I_g} \SoDiag^{i,a}(\cvector{v})
\end{equation*}
where the sum is taken modulo $2$ if $a$ is equal to $\tau(a)$.
\par
Recall that any element $g$ of $\pi$ can be written in the form given in
\eqref{eqn:pi-element}.
We define $\gamma_a(g)$ to be the exponent of $a$ when $g$ is written in
this form (that is, $\gamma_a(g)$ is equal to $i_a$ in
\eqref{eqn:pi-element}).
For a nanophrase, we define $\delta_a(\SoDiag^{i,b})$ by
\begin{equation}\label{eqn:su-singlecomp}
\delta_a(\SoDiag^{i,b}) = \sum_{g \in \pi - \{1\}} 
\sum_{\cvector{v} \in I_g} \SoDiag^{i,b}(\cvector{v})\gamma_a(g)
\end{equation} 
where the sum is taken modulo $2$ if $a$ is equal to $\tau(a)$.
In \cite{Turaev:Words}, Turaev gave necessary and sufficient conditions
for a self-linking function to be realizable as a self-linking function
of a nanoword.
These conditions translate to the condition that
\begin{equation*}
\delta_a(\SoDiag^{1,b}) + \delta_b(\SoDiag^{1,a}) = 0
\end{equation*}
for all $a$ and $b$ in $\alpha_0$ and
\begin{equation*}
\delta_a(\SoDiag^{1,a}) = 0
\end{equation*}
for all $a$ in $\alpha_0$.
\par
This gives some necessary conditions for a set of maps $\SoDiag^{i,a}$ to be
realizable as the $\SoDiag$ invariant of a nanopphrase.
In fact these conditions are also sufficient.
\begin{prop}
For each $a$ in $\alpha_0$,
let $\SoDiag^{i,a}$ be a map from $\pi^n - \{\cvector{1}\}$ to $\Z$,
 if $a$ is not equal to $\tau(a)$, and to $\cyclic{2}$ otherwise.
If, for each $i$, the maps defined in \eqref{eqn:su-singlecomp} satisfy
\begin{equation*}
\delta_a(\SoDiag^{i,b}) + \delta_b(\SoDiag^{i,a}) = 0
\end{equation*}
for all $a$ and $b$ in $\alpha_0$ and
\begin{equation*}
\delta_a(\SoDiag^{i,a}) = 0
\end{equation*}
for all $a$ in $\alpha_0$,
then the maps $\SoDiag^{i,a}$ are the $\SoDiag$ invariant of some nanophrase.
\end{prop}
\begin{proof}
This can be proved by combining the arguments in Theorem~6.3.1 of
 \cite{Turaev:Words} and Proposition~5.5 and Proposition~5.7 of
 \cite{Gibson:gauss-phrase}.
We omit the details.
\end{proof}
\begin{rem}
By a similar argument to that given in Proposition~5.7 of
 \cite{Gibson:gauss-phrase}, the linking matrix and $\SoDiag$ invariants can
 be shown to be independent.
\end{rem}
\subsection{$\SoDiag$ invariant and Fukunaga's $S_o$
  invariant}\label{subsec:fukunagaequiv} 
In \cite{Fukunaga:gen-app}, Fukunaga generalized our $S_o$ invariant for
Gauss phrases defined in \cite{Gibson:gauss-phrase}.
His invariant is also called $S_o$.
In this subsection we recall the definition of Fukunaga's $S_o$
invariant and show that it is equivalent to our $\SoDiag$ invariant.
\par
We fix a homotopy $(\alpha,\tau,S)$ with diagonal $S$.
An orbit of $\tau$ is called a \emph{free orbit} if the
orbit contains two elements of $\alpha$.
Otherwise, the orbit contains only one element and is called a
\emph{fixed orbit}.
Let $l$ be the number of free orbits of $\tau$ and $m$ be the number of
fixed orbits of $\tau$.
We denote the orbits of $\tau$ by $\widehat{a}_i$ where, for $i$ running
from $1$ to $l$, $\widehat{a}_i$ is a free orbit of $\tau$ and for $i$
running from $l+1$ to $m$, $\widehat{a}_i$ is a fixed orbit.
For each orbit $\widehat{a}_i$ of $\tau$ we fix a representative element
which we denote $a_i$.
The set of representative elements is an orientation of $\alpha$ which
we denote $\alpha_0$.
\par
Let $p$ be an $n$-component nanophrase.
For a letter $A$ in $p$, $\varepsilon(A)$ is defined as follows:
\begin{equation*}
\varepsilon(A) = 
\begin{cases}
 1 & \text{if } \xType{A}=a_j \text{ and } 1 \leq j \leq l+m \\
-1 & \text{if } \xType{A}=\tau(a_j) \text{ and } 1 \leq j \leq l.
\end{cases} 
\end{equation*}
\par
Let $K_{i,j}$ be $\Z$ if $i$ and $j$ are both less or equal to $l$ and
let $K_{i,j}$ be $\cyclic{2}$ otherwise.
Let $K$ be 
$K_{1,1} \times K_{1,2} \times \dotso \times K_{1,l+m} \times K_{2,1} \times \dotso \times K_{l+m,l+m}$.
Elements of $K$ are considered to be row vectors.
Let $\gvector{r_{s,t}}$ denote the element of $K$ which has a $1$ in the column
corresponding to $K_{s,t}$ and $0$ in every other column.
Let $\cvector{0}$ denote the zero vector in $K$.
\par
For a nanophrase $p$, let $\mathcal{A}(p)$ be the $\alpha$-alphabet
associated with $p$.
We now extend the definition of $n(A,B)$ given in
Section~\ref{subsec:suinvariant} to any two letters $A$ and $B$ in
$\mathcal{A}(p)$.
We define $n(A,B)$ to be $1$ if $p$ has the form
$\dotso A \dotso B \dotso A \dotso B \dotso$,
$n(A,B)$ to be $-1$ if $p$ has the form
$\dotso B \dotso A \dotso B \dotso A \dotso$
and $n(A,B)$ to be $0$ otherwise.
As before, $n(A,A)$ is defined to be $0$ for all $A$ in
$\mathcal{A}(p)$.
\par
For any letter $B$ in $\mathcal{A}(p)$, let $I_1(B)$ denote the index of
the component in which the first occurence of $B$ appears.
Similarly, let $I_2(B)$ denote the index of
the component in which the second occurence of $B$ appears.
\par
For any letter $A$ in $\mathcal{A}(p)$, Fukunaga defines
$\sigma_j(A,B)$ by
\begin{equation*}
\sigma_j(A,B) = 
\begin{cases}
\gvector{r_{s,t}} &
\text{if } \xType{A}\in \widehat{a}_s,\; \xType{B}=a_t,\; n(A,B)=1  \text{ and
 } I_2(B)=j, \\
\gvector{r_{s,t}} &
\text{if } \xType{A}\in \widehat{a}_s,\; \xType{B}=\tau(a_t),\; n(A,B)=-1
 \text{ and } I_1(B)=j, \\
-\gvector{r_{s,t}} &
\text{if } \xType{A}\in \widehat{a}_s,\; \xType{B}=\tau(a_t),\;
 n(A,B)=1 \text{ and } I_2(B)=j, \\
-\gvector{r_{s,t}} &
\text{if } \xType{A}\in \widehat{a}_s,\; \xType{B}=a_t,\;
 n(A,B)=-1 \text{ and } I_1(B)=j, \\
\cvector{0} &
\text{otherwise}.
\end{cases} 
\end{equation*}
\par
Now, for any letter $A$ in $\mathcal{A}(p)$, Fukunaga defines $l_j(A)$
by
\begin{equation*}
l_j(A) = \sum_{X \in \mathcal{A}(p)} \sigma_j(A,X)
\end{equation*}
and $l(A)$ as an $n$-tuple given by 
\begin{equation*}
l(A) = (l_1(A),l_2(A),\dotsc,l_n(A)).
\end{equation*}
Note that $l(A)$ is in $K^n$.
\par
For a vector $\cvector{v}$ in $K^n$, let $v_{i,s,t}$ be the value of the
column corresponding to $K_{s,t}$ in the $i$th component of
$\cvector{v}$.
Let $\gvector{0_n}$ be the zero vector in $K^n$.  
Now, for any vector $\cvector{v}$ in $K^n - \gvector{0_n}$, we define
the \emph{type} of $\cvector{v}$ as follows.
We say $\cvector{v}$ is of type~(i) if $v_{i,r,s}$ being non-zero
implies 
$r$ is less than or equal to $l$, for all $i$, $r$ and $s$. 
We say $\cvector{v}$ is of type~(ii) if $v_{i,r,s}$ being non-zero
implies 
$r$ is greater than $l$, for all $i$, $r$ and $s$. 
Otherwise, we say $\cvector{v}$ is of type~(iii).
\par
For an integer $i$ and a vector $\cvector{v}$ in $K^n$ we define
$\eta(i,\cvector{v})$ by
\begin{equation*}
\eta(i,\cvector{v}) = 
\sum_{A \in \mathcal{A}_{ii}(p),\; l(A)=\cvector{v}} \varepsilon(A)
\end{equation*}
Then Fukunaga defines a map $B_i$ from $K^n - \{\gvector{0_n}\}$ to
$\Z$ or $\cyclic{2}$, depending on the type of vector, as follows:
\begin{equation*}
B_i(\cvector{v}) = 
\begin{cases}
\eta(i,\cvector{v}) & \text{ if $\cvector{v}$ is of type~(i)}, \\
\eta(i,\cvector{v}) \mod 2 & \text{ if $\cvector{v}$ is of type~(ii)}, \\
0 & \text{otherwise.}
\end{cases}
\end{equation*}
Finally, $S_o(p)$ is the $n$-tuple of maps given by
\begin{equation*}
S_o(p) = (B_1,B_2,\dotsc,B_n).
\end{equation*}
In \cite{Fukunaga:gen-app}, Fukunaga proved that $S_o(p)$ is a homotopy
invariant of $p$.
\par
In order to show that Fukunaga's $S_o$ invariant and the $\SoDiag$
invariant are equivalent we need some preparation. 
\par
For any element $g$ in $\pi$, where $\pi$ is the group defined in
Section~\ref{subsec:linkingmatrix}, we can uniquely write $g$ in the
form
\begin{equation*}
g = \prod_{i=1}^{l+m}a_i^{c_i}
\end{equation*}
where $c_i$ is in $\Z$ if $i$ is less than or equal to $l$ and $c_i$ is
in $\cyclic{2}$ otherwise.
For each $i$ running from $1$ to $l$ we define a map $f_i$ from $\pi$ to
$\Z$.
For each $i$ running from $l+1$ to $l+m$ we define a map $f_i$ from $\pi$ to
$\cyclic{2}$.
In either case, $f_i(g)$ is defined to be $c_i$, the exponent of $a_i$
when $g$ is written in the form above.
\par
Now, for an element $a$ in $\alpha_0$ we define a set of maps $h_{a,i,j}$
from $\pi$ to $K$.
For $g$, an element of $\pi$, $h_{a,i,j}$ maps $g$ to $\cvector{u}$ where
the component of $\cvector{u}$ corresponding to $K_{s,t}$ is given by
\begin{equation*}
u_{s,t} = 
\begin{cases}
0 & \text{ if } a \notin \widehat{a}_s, \\
f_t(g) & \text{ if } a \in \widehat{a}_s \text{ and } j \leq i, \\
-f_t(g) & \text{ if } a \in \widehat{a}_s \text{ and } i < j.
\end{cases}
\end{equation*}  
Next we define a set of maps $h_{a,i}$ from $\pi^n$ to $K^n$.
The map $h_{a,i}$ takes $(g_1,g_2,\dotsc,g_n)$ to 
$(h_{a,i,1}(g_1),h_{a,i,2}(g_2),\dotsc,h_{a,i,n}(g_n))$.
\par
For a nanophrase $p$ and for some integer $i$, let $A$ be a letter in
$\mathcal{A}_{ii}(p)$. 
Then it follows from the definitions that Fukunaga's $l(A)$ is equal to
$h_{a,i}(\gvector{l_u(A)})$.
\par
Suppose that $\xType{A}$ is in $\widehat{a_r}$ for some $r$.
Recall that for a vector $\cvector{v}$ in $K^n$, $v_{j,s,t}$ is the
value of the column corresponding to $K_{s,t}$ in the $j$th component of
$\cvector{v}$.
Now observe that by definition, $l(A)_{j,s,t}$ is equal to zero if $s$
is not equal to $r$.
In other words, if $l(A)$ is non-zero, the orbit of $\tau$ which
contains $\xType{A}$ is implicitly recorded in $l(A)$.
\par
The following lemma shows that we can calculate $\SoDiag$ from
Fukunaga's invariant.
\begin{lem}\label{lem:sodiagfromfuku}
With the above notation, for any vector $\cvector{v}$ in 
$\pi^n - \{\cvector{1}\}$, we have
\begin{equation*}
\SoDiag^{i,a}(\cvector{v}) = B_i(h_{a,i}(\cvector{v})).
\end{equation*}
\end{lem}
\begin{proof}
When $a$ is not equal to $\tau(a)$, we have
\begin{align*}
\SoDiag^{i,a}(\cvector{v})
& = e_o^{i,a}(\cvector{v}) - e_o^{i,\tau(a)}(\cvector{v})\\
& = \sharp \lbrace X\in \mathcal{A}_{i,a} \; | \;
 \gvector{l_u(X)} = \cvector{v} \rbrace -
\sharp \lbrace X\in \mathcal{A}_{i,\tau(a)} \; | \;
 \gvector{l_u(X)} = \cvector{v} \rbrace \\
& = \sum_{A \in \mathcal{A}_{ii}(p),\; l(A)=h_{a,i}(\cvector{v})}
 \varepsilon(A) \\
& = B_i(h_{a,i}(\cvector{v})).
\end{align*}
The equality before last holds because, by the observation given before
 the statement of the lemma, $l(A)$ can only equal
 $h_{a,i}(\cvector{v})$ if $\xType{A}$ is equal to $a$ or $\tau(a)$.
The last equality holds because if $a$ is not equal to $\tau(a)$, and
 $\xType{A}$ equals $a$, $l(A)$ must be a type~(i) vector. 
\par
When $a$ is equal to $\tau(a)$, we have a similar calculation:
\begin{align*}
\SoDiag^{i,a}(\cvector{v})
& = e_o^{i,a}(\cvector{v}) \mod 2 \\
& = \sharp \lbrace X\in \mathcal{A}_{i,a} \; | \;
 \gvector{l_u(X)} = \cvector{v} \rbrace \mod 2 \\
& = \sum_{A \in \mathcal{A}_{ii}(p),\; l(A)=h_{a,i}(\cvector{v})}
 \varepsilon(A) \mod 2 \\
& = B_i(h_{a,i}(\cvector{v})).
\end{align*}
The last equality holds because if $a$ is equal to $\tau(a)$, and
 $\xType{A}$ equals $a$, $l(A)$ must be a type~(ii) vector. 
\end{proof}
Let $\mathcal{K}$ denote the subset of $K^n - \{\gvector{0_n}\}$ given
by 
\begin{equation*}
\mathcal{K} = \{ \cvector{v} \in K^n - \{\gvector{0_n}\} \; | \;
\exists r \in \Z \text{ such that } \forall j,s,t, \;
 v_{j,s,t} \neq 0 \Rightarrow s = r\; \}.
\end{equation*}
By the observation given before Lemma~\ref{lem:sodiagfromfuku} it is
clear that for any nanophrase $p$ and any
letter $A$ in $\mathcal{A}_{ii}(p)$, either $l(A)$ is
$\gvector{0_n}$ or $l(A)$ is in $\mathcal{K}$.
\par
We define maps $\kappa_i$ from $\mathcal{K}$ to $\pi_n$.
Let $\cvector{v}$ be a vector in $\mathcal{K}$.
By definition there exists an $r$ such that $v_{j,s,t}$ is non-zero
implies $s$ equals $r$.
Denote that $r$ by $r(\cvector{v})$.
Now let $\cvector{u}$ be $\kappa_i(\cvector{v})$ for some integer $i$.
Then $u_j$, the $j$th component of $\cvector{u}$, is given by
\begin{equation*}
u_j = 
\begin{cases}
\prod_{a_t \in \alpha_0} a_t^{v_{j,r,t}} & \text{if } j \geq i, \\
\prod_{a_t \in \alpha_0} a_t^{-v_{j,r,t}} & \text{if } j < i.
\end{cases}
\end{equation*}
\par
The following lemma shows that we can calculate Fukunaga's invariant
from $\SoDiag$.
\begin{lem}\label{lem:fukufromsodiag}
Given a vector $\cvector{v}$ in 
$K^n - \{\gvector{0_n}\}$, write $a$ for $a_{r(\cvector{v})}$.
Then we have
\begin{equation*}
B_i(\cvector{v}) = \SoDiag^{i,a}(\kappa_i(\cvector{v})).
\end{equation*}
\end{lem}
\begin{proof}
If $\cvector{v}$ is of type~(i) then $a$ is not equal to $\tau(a)$.
In this case we have
\begin{align*}
B_i(\cvector{v})
& = \sum_{A \in \mathcal{A}_{ii}(p),\; l(A)=\cvector{v}} \varepsilon(A) \\
& = \sharp \lbrace X\in \mathcal{A}_{i,a} \; | \;
 l(A) = \cvector{v} \rbrace -
\sharp \lbrace X\in \mathcal{A}_{i,\tau(a)} \; | \;
 l(A) = \cvector{v} \rbrace \\
& = \sharp \lbrace X\in \mathcal{A}_{i,a} \; | \;
 \gvector{l_u(A)} = \kappa_i(\cvector{v}) \rbrace -
\sharp \lbrace X\in \mathcal{A}_{i,\tau(a)} \; | \;
 \gvector{l_u(A)} = \kappa_i(\cvector{v}) \rbrace \\
& = e_o^{i,a}(\kappa_i(\cvector{v})) -
 e_o^{i,\tau(a)}(\kappa_i(\cvector{v})) \\
& = \SoDiag^{i,a}(\kappa_i(\cvector{v})).
\end{align*}
If $\cvector{v}$ is of type~(ii) then $a$ is equal to $\tau(a)$.
In this case we have a similar calculation:
\begin{align*}
B_i(\cvector{v})
& = \sum_{A \in \mathcal{A}_{ii}(p),\; l(A)=\cvector{v}} \varepsilon(A)
\mod 2
 \\
& = \sharp \lbrace X\in \mathcal{A}_{i,a} \; | \;
 l(A) = \cvector{v} \rbrace \mod 2 \\
& = \sharp \lbrace X\in \mathcal{A}_{i,a} \; | \;
 \gvector{l_u(A)} = \kappa_i(\cvector{v}) \rbrace \mod 2 \\
& = e_o^{i,a}(\kappa_i(\cvector{v})) \mod 2 \\
& = \SoDiag^{i,a}(\kappa_i(\cvector{v})).
\end{align*}
\end{proof}
Combining Lemma~\ref{lem:sodiagfromfuku} and
Lemma~\ref{lem:fukufromsodiag} proves the following proposition.
\begin{prop}\label{prop:fukunagaequiv}
Fukunaga's $S_o$ invariant is equivalent to the $\SoDiag$ invariant.
\end{prop}
\begin{rem}
In \cite{Fukunaga:nanophrases2}, Fukunaga defined the $T$ invariant for
nanophrases which is invariant under homotopies with diagonal $S$.
In \cite{Fukunaga:gen-app}, Fukunaga showed that his generalized $S_o$
 invariant is strictly stronger than his $T$ invariant.
In particular, in Proposition~5.4 of \cite{Fukunaga:gen-app}, Fukunaga
 showed how to calculate the $T$ invariant from the $S_o$ invariant.
As our $\SoDiag$ invariant is equivalent to Fukunaga's $S_o$ invariant, 
 $\SoDiag$ is strictly stronger than the $T$ invariant.
\end{rem}
\section{Nanomultiphrases}\label{sec:nanophrase-inv}
An $n$-phrase \emph{multiphrase} on an alphabet $\mathcal{A}$ is a
sequence of $n$ phrases on $\mathcal{A}$.
The concepts of \emph{Gauss multiphrase} and \emph{nanomultiphrase} are
defined in an analogous way to those of Gauss phrase and nanophrase.
When writing a multiphrase we use the symbol `$||$' to separate phrases.
For example $A|B||AC||D|B|CD$ is a $3$-phrase Gauss multiphrase on
$\{A,B,C,D\}$.
Note that the multiphrases $A||B$ and $A|\trivial|B$ are different,
because an empty component in a phrase is always written $\trivial$.
So $A||B$ is a $2$-phrase multiphrase where each phrase has a
single component.
The multiphrase $A|\trivial|B$ is a $1$-phrase multiphrase where the
single phrase has $3$ components.
\par
The $0$-component phrase $\zerocomp$ may also appear as a phrase in a
nanomultiphrase.
There is a unique $0$-phrase nanomultiphrase which is written
$\zerophrase$.
\par
The only phrase of a $1$-phrase nanomultiphrase is necessarily a
nanophrase.
Thus we can identify $1$-phrase nanomultiphrase with nanophrases.
\par
We define isomorphism of nanomultiphrases and homotopy moves on
nanomultiphrases in an analogous way to nanophrases.
Homotopy of nanomultiphrases is then the equivalence relation generated
by isomorphisms and homotopy moves.
We note that the number of phrases and the number of components in each
phrase are invariant under homotopy.
\par
Let $p_1$ be an $n_1$-component phrase and $p_2$ be an $n_2$-component
phrase. 
Their \emph{concatenation}, written $p_1|p_2$, is the
$(n_1+n_2)$-component phrase consisting
of the components of $p_1$ followed by the components of $p_2$.
\par
Let $\mathcal{M}(\alpha)$ be the set of nanomultiphrases over $\alpha$.
There exists a natural map $\mu$ from $\mathcal{M}(\alpha)$ to
$\mathcal{P}(\alpha)$ where a nanomultiphrase $m$ is mapped to a
nanophrase by concatenating the phrases of $m$ to make a single
phrase.
The following lemma is clear from the definitions.
\begin{lem}
Let $m_1$ and $m_2$ be nanomultiphrases.
If $m_1 \sim m_2$, then $\mu(m_1) \sim \mu(m_2)$ as nanophrases.
\end{lem}
Let $m$ be a nanomultiphrase.
Then we define the $i$th component of $m$ to be the $i$th component of
$\mu(m)$.
We define $\nc(m)$ to be $\nc(\mu(m))$, which is equal to the sum of the
number of components appearing in each phrase of $m$.
For example, $\nc(A|B||AC||D|B|CD)$ is $6$.
\par
Recall that in Section~\ref{sec:nanophrases} we defined the
\emph{concatenating map} $\chi$ from $\mathcal{P}(\alpha)$ to
$\mathcal{N}(\alpha)$.
We extend $\chi$ to be a map from $\mathcal{M}(\alpha)$ to
$\mathcal{P}(\alpha)$ as follows.
Let $m$ be an $n$-phrase nanomultiphrase.
Then $\chi(m)$ is the $n$-component nanophrase $p$ where the $i$th
component of $p$ is the word given by concatenating the components of
the $i$th phrase of $m$.
For example, $\chi(A|B||AC||D|B|CD)$ is $AB|AC|DBCD$.
\par
We fix a homotopy data triple $(\alpha,\tau,S)$ and let its prime
factors be denoted by $(\alpha_i,\tau_i,S_i)$ for $i$
running from $1$ to $k$ for some $k$.
Let $\mathcal{M}_A(\alpha)$ be the set of pairs $(m,\theta)$,
where $m$ is a nanomultiphrase over $\alpha$
and $\theta$ is a map from $\widehat{\nc(m)}$ to
$\hat{k}$ such that, for all $i$, $\theta(i) = \theta(i+1)$ implies the
$i$th and $(i+1)$th components of $m$ belong to different phrases of
$m$. 
We extend the definition of isomorphism and homotopy moves of
nanophrases to $\mathcal{M}_A(\alpha)$ by saying that 
$(m,\theta) \sim (m^\prime,\theta^\prime)$ if $p \sim p^\prime$ and
$\theta = \theta^\prime$.
\par
Let $\mathcal{M}_R(\alpha)$ be the subset of $\mathcal{M}_A(\alpha)$
consisting of pairs 
$(m,\theta)$ such that for every letter $X$ appearing in the $i$th
component of $m$, $\xType{X}$ is in $\alpha_{\theta(i)}$, for all $i$.
We define the equivalence relation $\sim_K$ on $\mathcal{M}_R(\alpha)$ as
follows. 
We say $(m,\theta) \sim_K (m^\prime,\theta)$
if there exists a sequence of
elements of $\mathcal{M}_R(\alpha)$, $(m_i,\theta)$ for $i$ running from
$0$ to $r$ for some $r$, such that $m_0$ is $m$, $m_r$ is $m^\prime$
and, for each $i$, $m_i$ is related to $m_{i+1}$ by a single homotopy
move or an isomorphism. 
\par
Reductions and augmentations are defined on $\mathcal{M}_R(\alpha)$ in an
analogous way to how we defined them for $\mathcal{P}_R(\alpha)$.
However, we do not allow concatenating reductions that would concatenate
components coming from different phrases.
So if the first or last component of a phrase is reducible, the
reduction is made by a simple reduction, irrespective of $\theta$.
Let $\sim_M$ be the equivalence relation on $\mathcal{P}_M$ generated by
$\sim_K$, reductions and augmentations.
\par
We extend the map $\psi$ to be a map from $\mathcal{P}(\alpha)$ to
$\mathcal{M}_R(\alpha)$ as follows.
The $0$-component nanophrase $\zerocomp$ is mapped to $\zerophrase$.
Any other element $p$ of $\mathcal{P}(\alpha)$ is an $n$-component
nanophrase in $\mathcal{P}(\alpha)$ with $n$ greater than $0$.
We map each component $w_i$ of $p$ to a phrase.
If $w_i$ is $\trivial$ then we map $w_i$ to $\zerocomp$.
Otherwise, there is a unique integer $r > 0$, a unique sequence of words
$x_1, x_2, \dotsc, x_r$ and a unique map
$\theta_i$ from $\hat{r}$ to $\hat{k}$ such that
\begin{enumerate}
\item
each $x_j$ is not $\trivial$;
\item
$w_i$ is equal to $x_1x_2 \dotso x_r$;
\item
for all $j$, if $X$ appears in $w_j$, $\xType{X}$ is in
     $\alpha_{\theta_i(j)}$;
\item
for any $j$ from $1$ to $r-1$, $\theta_i(j)$ is not equal to
     $\theta(j+1)$.
\end{enumerate}
Then $w_i$ maps to the phrase $x_1|x_2|\dotsc|x_r$.
Write $p_i$ for the image of $w_i$.
Then $p$ is mapped to $p_1||p_2||\dotsc||p_n$ which we label $m$.
We define $\theta$ to be the map such that if $X$ is a letter in the
$i$th component of $m$ and $\xType{X}$ is in $\alpha_j$, then
$\theta(i)$ is equal to $j$.
We define $\psi(p)$ to be $(m,\theta)$.
\par
We have the following lemma which corresponds to
Lemma~\ref{lem:psiinvariance}.
\begin{lem}\label{lem:multi-psiinvariance}
Let $p$ and $p^\prime$ be elements of $\mathcal{P}(\alpha)$.
If $p \sim p^\prime$, then $\psi(p) \sim_M \psi(p^\prime)$.
\end{lem}
\begin{proof}
We omit the proof. Is it is similar to that of
 Lemma~\ref{lem:psiinvariance}.
\end{proof} 
The following theorem is analogous to Theorem
\ref{thm:bijection}.
\begin{thm}
There is a bijection between the homotopy classes of
 $\mathcal{P}(\alpha)$ and the equivalence classes of $\mathcal{M}_R(\alpha)$
 under $\sim_M$.
In other words
\begin{equation*}
\mathcal{P}(\alpha) / \sim \quad \cong \quad \mathcal{M}_R(\alpha) / \sim_M.
\end{equation*} 
\end{thm}
We define $\mathcal{K}_M(\alpha)$ to be the set of equivalence classes of
$\mathcal{M}_R(\alpha)$ under $\sim_K$.
Then we can define the concept of $i$-reduction of an element of
$\mathcal{K}_M(\alpha)$ as we did for $\mathcal{K}(\alpha)$.
The proof that this concept is well-defined is almost identical to that
of the case of $\mathcal{K}(\alpha)$ (Lemma~\ref{lem:reduction-wd}).
\par
It is easy to check that the Confluence condition proved in
Lemma~\ref{lem:confluence} also holds for $\mathcal{K}_M(\alpha)$.
As the Finiteness condition also holds, we can use Newman's Diamond
Lemma to get the following proposition corresponding to
Proposition~\ref{prop:reduction-wd}.
\begin{prop}\label{prop:multi-reduction-wd}
Each equivalence class of $\mathcal{K}_M(\alpha)$ under $\sim_M$ contains
 exactly one reduced element.
\end{prop} 
For an element $p$ of $\mathcal{P}$ we denote by $R_K(p)$ the reduced
element in $\mathcal{K}_M(\alpha)$ corresponding to the element of
$\mathcal{K}_M(\alpha)$ containing $\psi(p)$.
The following theorem corresponds to Theorem~\ref{thm:invariance}.
\begin{thm}\label{thm:multi-invariance}
For a nanophrase $p$ in $\mathcal{P}(\alpha)$, $R_K(p)$ is a homotopy
 invariant of $p$. 
\end{thm}
We extend the map $s_i$ to elements of $\mathcal{M}_R(\alpha)$ in a
natural way.
For $(m,\theta)$ in $\mathcal{M}_R(\alpha)$ we define $s_i((m,\theta))$
to be the nanomultiphrase derived from $(m,\theta)$ by deleting all
components for which the index of the component is mapped by $\theta$ to
some integer other than $i$.
By definition $s_i((m,\theta))$ contains the same number of phrases as
$(m,\theta)$.
\par
Let $(m,\theta)$ be an element of $R_K(p)$.
We write $\theta_R(p)$ for $\theta$ which is an invariant of $p$.
We also define $P_{R,i}(p)$ to be $s_i((m,\theta))$.
The nanomultiphrase $P_{R,i}(p)$ is a nanomultiphrase in
$\mathcal{M}(\alpha_i)$ and is, modulo $\sim_i$, an invariant of $p$.
Taken together, the map $\theta_R(p)$ and the set of nanomultiphrases
$P_{R,i}(p)$ give a complete invariant for nanophrases.
\par
Definitions of rank and homotopy rank naturally extend to
nanomultiphrases and we use the same notation as we did for nanowords
and nanophrases.
Following the arguments about rank and homotopy rank in
Section~\ref{sec:nanoword-inv}, it is easy to prove the following
theorem which corresponds to Theorem~\ref{thm:hr}.
\begin{thm}
Let $(\alpha,\tau,S)$ be a composite homotopy data
 triple and let its prime factors be denoted by
 $(\alpha_i,\tau_i,S_i)$.
Let $p$ be a nanophrase over $\alpha$.
Then
\begin{equation}
\hr(p) = \sum_i \hr(P_{R,i}(p)).
\end{equation}
\end{thm}
We finish by giving the following theorem.
\begin{thm}\label{thm:phrasedecidability}
Let $(\alpha,\tau,S)$ be a composite homotopy data
 triple and let its prime factors be denoted by
 $(\alpha_i,\tau_i,S_i)$.
Suppose, for each $i$, the homotopy given by
 $(\alpha_i,\tau_i,S_i)$ is reduction decidable and
 equality decidable.
Then the homotopy given by $(\alpha,\tau,S)$ is also
 reduction decidable and equality decidable.
\end{thm}
\begin{proof}
The fact that $(\alpha,\tau,S)$ is equality decidable is
 analogous to Theorem~\ref{thm:decidability} and can be proved
 similarly.
\par
Given a nanophrase $p$ in $\mathcal{P}$, $p$ is $j$-reducible if and
 only if all the components in the $j$th phrase of $\psi(p)$ are
 reducible.
Since the homotopies given by $(\alpha_i,\tau_i,S_i)$
 are all reduction decidable, there is a finite time algorithm which
 determines the reducibility of all these components.
Thus the homotopy given by $(\alpha,\tau,S)$ is reduction decidable.
\end{proof}
\bibliography{mrabbrev,nwdecomp}
\bibliographystyle{hamsplain}
\end{document}